\documentclass[dvips,aap,preprint]{imsart}

\RequirePackage[OT1]{fontenc}
\RequirePackage{amsthm,amsmath}
\RequirePackage[colorlinks]{hyperref}
\RequirePackage{hypernat}
\usepackage{amssymb}
\usepackage{paralist}
\usepackage{mathrsfs}

\startlocaldefs
\numberwithin{equation}{section}
\theoremstyle{plain}
\newtheorem{thm}{Theorem}[section]
\newtheorem{prop}{Proposition}[section]

\newtheorem{lem}{Lemma}[section]
\endlocaldefs

\newcommand{\aaa}{C_{1}}
\newcommand{\bbb}{C_{2}}
\newcommand{\ccc}{C_{3}}
\newcommand{\ddd}{C_{4}}
\newcommand{\eee}{C_{5}}

\newcommand{\epsa}{\epsilon_{1}}
\newcommand{\epsagg}{\epsilon_{2}}
\newcommand{\epsb}{\epsilon_{3}}
\newcommand{\epsc}{\epsilon_{9}}
\newcommand{\epsd}{\epsilon_{4}}
\newcommand{\epsf}{\epsilon_{10}}
\newcommand{\epsg}{\epsilon_{2}}
\newcommand{\epsk}{\epsilon_{5}}
\newcommand{\epsp}{\epsilon_{6}}
\newcommand{\epsq}{\epsilon_{7}}
\newcommand{\epsr}{\epsilon_{8}}

\newcommand{\ssb}{D_{1}}
\newcommand{\ssc}{C_{1}}
\newcommand{\ssd}{D_{2}}
\newcommand{\sse}{D_{3}}

\newcommand{\ssg}{D_{5}}
\newcommand{\ssh}{D_{4}}
\newcommand{\ssi}{C_{1}}
\newcommand{\ssj}{C_{4}}
\newcommand{\ssk}{C_{1}}
\newcommand{\ssl}{C_{2}}
\newcommand{\ssm}{C_{3}}

\newcommand{\ffc}{F_{3}}
\newcommand{\ffd}{F_{4}}
\newcommand{\ffe}{F_{5}}
\newcommand{\ffg}{F_{6}}
\newcommand{\ffh}{F_{7}}
\newcommand{\ffi}{F_{8}}
\newcommand{\ffj}{F_{9}}
\newcommand{\ffk}{F_{10}}

\newcommand{\gga}{G_{1}}
\newcommand{\ggb}{G_{2}}
\newcommand{\ggc}{G_{3}}
\newcommand{\ggd}{G_{4}}
\newcommand{\gge}{G_{5}}

\newcommand{\hha}{H_{1}}
\newcommand{\hhb}{H_{2}}
\newcommand{\hhc}{H_{3}}
\newcommand{\hhd}{H_{4}}
\newcommand{\hhe}{H_{5}}

\begin{document}

\begin{frontmatter}
\title{ A Positive Recurrent Reflecting Brownian Motion with Divergent Fluid Path}
\runtitle{Recurrent SRBM with Divergent Fluid Path}
\begin{aug}
\author{\fnms{Maury} \snm{Bramson}\thanksref{t1}\ead[label=e1]
{bramson@math.umn.edu}}

\thankstext{t1}{Supported in part by NSF Grant CCF-0729537}
\runauthor{Maury Bramson}

\affiliation{University of Minnesota, Twin Cities} 

\address{University of Minnesota\\
Twin Cities Campus\\
School of Mathematics\\
Institute of Technology\\
127 Vincent Hall\\
206 Church Street S.E.\\
Minneapolis, Minnesota 55455\\
\printead{e1}\\
\phantom{E-mail:\ }}

\end{aug}

\begin{abstract}
Semimartingale reflecting Brownian motions (SRBMs) are diffusion processes with state space the 
$d$-dimensional nonnegative orthant, in the interior of which the processes evolve according to a Brownian 
motion, and that reflect against the boundary in a specified manner.  The data for such a process are a drift vector
$\theta$, a nonsingular $d\times d$ covariance matrix $\Sigma$, and a $d\times d$ reflection matrix $R$.  
A standard problem is to determine under what conditions the process is positive recurrent.
Necessary and sufficient conditions for positive recurrence are easy to formulate for $d=2$, but not for $d>2$.

Associated with the pair $(\theta,R)$ are fluid paths, which are solutions of deterministic equations corresponding
to the random equations of the SRBM.  A standard result of Dupuis and Williams \cite{DW} states that when every fluid path 
associated with the SRBM is attracted to the origin, the SRBM is positive recurrent.  Employing this result, El 
Kharroubi et al. \cite{EK1, EK2} gave sufficient conditions on $(\theta,\Sigma,R)$ for positive recurrence for
$d=3$; Bramson et al. \cite{BDH} showed that these conditions are, in fact, necessary.

Relatively little is known about the recurrence behavior of SRBMs for $d>3$.  This pertains, in particular, to 
necessary conditions for positive recurrence.  Here, we provide a family of examples, in $d=6$, with
$\theta=(-1,-1,\ldots,-1)^T$, $\Sigma=I$ and appropriate $R$, that are positive recurrent, but for which a linear fluid
path diverges to infinity.  These examples show in particular that, for $d\ge6$, the converse of the Dupuis-Williams
result does not hold.
\end{abstract}



\end{frontmatter}

\section{Introduction}
This paper is concerned with the class of $d$-dimensional diffusion processes known as semimartingale reflecting
Brownian motions (SRBMs).  Such processes arise as approximations for open $d$-station queueing networks 
(see, e.g., Harrison and Nguyen \cite{HN} and Williams \cite{W1, W2}).  The state space for a process 
$Z=\{Z(t), t\ge0\}$ in this class is $S=\mathbb{R}^{d}_{+}$, the nonnegative orthant.  The data of the process consists
of a drift vector $\theta$, a nonsingular covariance matrix $\Sigma$, and a $d\times d$ reflection matrix $R$ that
specifies the boundary behavior.  In the interior of the orthant, $Z(\cdot)$ 
behaves as an ordinary Brownian motion
with parameters $\theta$ and $\Sigma$ and, roughly speaking, $Z(\cdot)$ is pushed in direction $R^k$ whenever
the boundary $\{z\in S: z_k = 0\}$ is hit, for $k=1,\ldots,d$, where $R^k$ is the $k^{\text{th}}$ column of $R$.  
The process is Feller \cite{TW} and so is strong Markov. 

A precise description for $Z(\cdot)$ is given by 
\begin{equation}
\label{eq1.2.1}
Z(t) = Z(0) +B(t) + \theta t + RY(t), \qquad t\ge0,
\end{equation}
where $B(\cdot)$ is an unconstrained Brownian motion with covariance 
vector $\Sigma$ and no drift, with $B(0)=0$,
and $Y(\cdot)$ is a $d$-dimensional process with components $Y_{1}(\cdot)$,\ldots,$Y_{d}(\cdot)$ such that
\begin{align}
\label{eq1.2.2}
& Y(\cdot) \text{ is continuous and nondecreasing, with } Y(0)=0, \\
\label{eq1.2.3}
& Y_{k}(\cdot) \text{ only increases at times t at which } Z_{k}(t)=0, \quad k=1,\ldots,d, \\
\label{eq1.2.4}
& Z(t) \in S \text{ for all } t\ge0.
\end{align}
(Display (\ref{eq1.2.3}) means that $Y_k(t_2) > Y_k(t_1)$, for $t_2 > t_1$, 
implies $Z_k(t)=0$ at some $t\in[t_1,t_2]$.)
For a SRBM with data $(\theta,\Sigma,R)$ to exist, it is necessary and sufficient that $R$ be
\emph{completely-$S$}.  Completely-$S$ means that each principal submatrix $R^{\prime}$ is an \emph{$S$-matrix},
that is, for some $w\ge0$, $R^{\prime}w >  0$ holds.  The complete definition and basic properties of $Z(\cdot)$ are
reviewed in Appendix A of Bramson et al. \cite{BDH}.

A SRBM is said to be \emph{positive recurrent} if the expected time to hit an arbitrary open neighborhood of the origin
is finite for every starting state.  A necessary and sufficient condition for positive recurrence, for $d=2$, is that
\begin{equation}
\label{eq1.3.1}
R \text{ is nonsingular with } R^{-1}\theta < 0
\end{equation}
and that $R$ is a $P$-matrix (El Kharroubi et al. \cite{EK1}).  (That is, each principal submatrix of $R$ has a 
positive determinant.)  Necessary and sufficient conditions, for $d=3$, are known, but are more complicated.  El Kharroubi
et al. \cite{EK2} gave sufficient conditions; Bramson et al. \cite{BDH} showed these conditions are necessary.
Another proof of the sufficiency of these conditions was recently given in Dai and Harrison \cite{DH}.  In the special case where
$R$ is an $M$-matrix, (\ref{eq1.3.1}) is necessary and sufficient for positive recurrence in all $d$ (Harrison and Williams 
\cite{HW}); (\ref{eq1.3.1}) is always necessary for positive recurrence (\cite{EK1}).

Associated with the parameters $\theta$ and $R$ are \emph{fluid paths}, which are solutions of deterministic equations
corresponding to (\ref{eq1.2.1})--(\ref{eq1.2.4}).  More precisely, a fluid path is a pair of continuous functions
$y,z: [0,\infty) \rightarrow \mathbb{R}^d$ that satisfy
\begin{align}
\label{eq1.4.1}
& z(t) = z(0) + \theta t + Ry(t) \text{ for all } t\ge 0, \\
\label{eq1.4.3}
& y(\cdot) \text{ is continuous and nondecreasing, with } y(0)=0, \\
\label{eq1.4.4}
& y_{k}(\cdot) \text{ only increases at times } t \text{ at which } z_{k}(t)=0, \quad k=1,\ldots,d, \\
\label{eq1.4.2}
& z(t) \in S \text{ for all } t\ge 0. 
\end{align}
A fluid path $(y,z)$ is \emph{attracted to the origin} if $z(t)\rightarrow 0$ as $t\rightarrow \infty$; it is
\emph{divergent} if $|z(t)| \rightarrow \infty$ as $t\rightarrow \infty$ 
(where $|u| \stackrel{\text{def}}{=} \Sigma_i |u_i|$, for $u = (u_i) \in \mathbb{R}^d$). 

The following result gives a sufficient condition for positive recurrence of an SRBM in terms of the associated
fluid paths.
\begin{thm} [Dupuis-Williams]
\label{thm1.4.3}
Let $Z(\cdot)$ be a $d$-dimensional SRBM with data $(\theta,\Sigma,R)$.  If every fluid path associated with
$(\theta,R)$ is attracted to the origin, then $Z(\cdot)$ is positive recurrent.%
\end{thm}

Theorem \ref{thm1.4.3} provides an important ingredient for demonstrating the sufficiency of the conditions in 
\cite{EK2} for positive recurrence of an SRBM, for $d=3$, that were alluded to above.  An open question is whether
a converse of Theorem \ref{thm1.4.3} holds for $d>3$, that is, whether $Z(\cdot)$ positive recurrent implies that
every fluid path is attracted to the origin.

A fluid path $(y,z)$ is \emph{linear} if $y(t)=ut$ and $z(t)=vt$ for given vectors $u,v\ge0$.  ($u\ge0$ 
means $u_i\ge0$ for $i=1,\ldots,d$.)  When $y(\cdot)$ and $z(\cdot)$ are linear,
the fluid path properties (\ref{eq1.4.1})-(\ref{eq1.4.2}) can be expressed as solutions of the linear 
complementarity problem
\begin{equation}
\label{eq1.5.1}
u,v\ge0, \qquad v = \theta + Ru, \qquad  u\cdot v = 0,
\end{equation}
%
%
where $u\cdot v \stackrel{\text{def}}{=} \sum_i u_{i}v_{i}$.  A solution $(u,v)$ of (\ref{eq1.5.1})
is \emph{stable} if $v=0$ and \emph{divergent} otherwise.  It is \emph{nondegenerate} if $u$ and $v$ together have
exactly $d$ positive components, and it is \emph{degenerate} otherwise.  It is easy to see that, for a converse to
Theorem \ref{thm1.4.3} to hold, all linear fluid paths associated with a positive recurrent SRBM must be stable.

In this article, we provide a family of examples, in $d=6$, for which the SRBM is positive recurrent, yet possesses
a divergent linear fluid path.  We set
\begin{equation}
\label{eq1.6.0}
\theta= (-1,-1,\ldots,-1)^T, \qquad  \Sigma= I 
\end{equation}
(where ``$\,^T\,$" denotes the transpose), and denote by $R$ the $6\times 6$ matrix with
\begin{equation}
\label{eq1.6.1}
R=J_1+J_2,
\end{equation}
where $J_1$ satisfies $(J_1)_{i,j} = 1$, for $i,j=1,\ldots,6$, and
\begin{equation}
\label{eq1.6.2}
J_2 =
\begin{bmatrix}  0&\delta_2&\delta_2&\delta_2&\delta_2&-\delta_4\\
0&0&-\delta_3&-\delta_3&-\delta_3&-\delta_3\\
0&-\delta_3&0&-\delta_3&-\delta_3&-\delta_3\\
0&-\delta_3&-\delta_3&0&-\delta_3&-\delta_3\\
0&-\delta_3&-\delta_3&-\delta_3&0&-\delta_3\\
\delta_1&-\delta_3&-\delta_3&-\delta_3&-\delta_3&0\\
\end{bmatrix}
.
\end{equation}
Here, we assume that $\delta_i>0$, $i=1,\ldots,4$, with
\begin{equation}
\label{eq1.6.3}
\delta_2 + \delta_3 \le \tfrac{1}{6}\delta_4
\end{equation}
and
\begin{equation}
\label{eq1.6.4}
\delta_1 \le \delta_3 \le .1, \qquad  \delta_4 < 1. 
\end{equation}
One can, for example, choose
\begin{equation}
\label{eq1.7.1}
\delta_1 = \delta_2 = \delta_3 = .05, \qquad \delta_4 = .6.
\end{equation}

The matrix $R$ has been chosen so that $R_{i,i} = 1$ for $i=1,\ldots,6$.  The roles of the coordinates
$i=2,\ldots,5$ with respect to $R$ are indistinguishable, and the role of $i=6$ differs from those of
$i=2,\ldots,5$ only in its interaction with the coordinate $i=1$ through $R_{1,6}$ and $R_{6,1}$.  Since
all entries of $R$ are positive, it is immediate that $R$ is completely-$S$.  The role of the relations in
(\ref{eq1.6.3})--(\ref{eq1.6.4}) will be explained in the next subsection.

The main result in this article is the following theorem. 
\begin{thm} 
\label{thm1.7.2}
Let $Z(\cdot)$ denote the SRBM with $\theta=(-1,-1,\ldots,-1)^T$, $\Sigma = I$ and $R$ satisfying 
(\ref{eq1.6.1})--(\ref{eq1.6.4}).  Then $Z(\cdot)$ is positive recurrent, but possesses a divergent linear fluid path.
\end{thm}

One can check that $(u,v)$, with $u=e_1$ and $v = \delta_{1}e_{6}$, defines a divergent linear fluid path  
($e_i$ denotes the $i^{\text{th}}$ unit vector).
Since $u$ and $v$ together have a total of two positive components, the fluid
path is degenerate.  (Related divergent fluid paths are easy to construct: for example, $(y,z)$ with 
$y(t) = e_{1}t$ and $z(t) = \sum_{k=2}^{5}e_k +\delta_{1}e_{6}t$.)
In order to demonstrate Theorem \ref{thm1.7.2}, it suffices to show $Z(\cdot)$ is
positive recurrent. 

Similar examples exist that satisfy the analog of Theorem \ref{thm1.7.2}, 
but with $d>6$.  One can construct such examples by inserting additional
coordinates $Z_i(\cdot)$ that are independent of $Z_1(\cdot),\ldots, Z_6(\cdot)$, with $\theta_i = -1$ 
and $R_{i,i} = 1$.   

In the remainder of the section, we summarize how the matriz $R$ affects the evolution of $Z(\cdot)$
and leads to its positive recurrence.  We also outline the rest of the paper. 

\subsection*{Sketch of positive recurrence}
The reflection matrix $R$ that we have chosen has the following properties, which we will use in the
next three paragraphs. For $\theta$ given by 
(\ref{eq1.6.0}), all of the coordinates $Z_{k}(\cdot), k=1,\ldots,6$, have drift $-1$, 
which is compensated for by $R$, which pushes a coordinate away from $0$ whenever any of the coordinates
is being reflected there.  (Although the motion induced by $R$ is not absolutely continuous, we will also refer
to it as ``drift" here.)  Because of the choice of $\theta$, for $k,k^{\prime}=2,\ldots,5$, 
\begin{equation*}
Z_{k^{\prime}}(\cdot) - Z_{k}(\cdot) \text{ has no drift} 
\end{equation*}
except when one of the coordinates is being reflected; when the coordinate $k$ is reflected, the difference has
negative drift because of the term $\delta_3$ in $J_2$.  Also, for $k=2,\ldots,5$,
$Z_6(\cdot) - Z_{k}(\cdot)$  has no drift  
except when $Z_k(\cdot)$ is reflected, in which case the difference has negative drift, or when $Z_6(\cdot)$ or 
$Z_1(\cdot)$ is reflected, in which case it has positive drift, the last case occurring because of the
term $\delta_1$.  On the other hand, when the first coordinate is being reflected, for $k=2,\ldots,5$,
$$ Z_1(\cdot) - Z_{k}(\cdot) \text{ has no drift} $$
and, when one of the other four coordinates $k=2,\ldots,5$ is being reflected, 
the difference has positive drift because of the term $\delta_2$ in $J_2$.  
But, when $Z_{6}(\cdot)$ is reflected, 
the difference acquires a negative drift because of the term $\delta_4$ in $J_2$ and (\ref{eq1.6.3}).

The process $Z(\cdot)$ is positive recurrent, although its deterministic analog $z(\cdot)$ possesses a 
divergent linear fluid path in the direction $e_6$
when $u=e_1$.  This difference in behavior occurs due to the following interaction between the different coordinates 
of $Z(\cdot)$.  When $Z_{1}(\cdot)$ is close to $0$
(for instance, when $Z_{k}(\cdot)$, $k=2,\ldots,5$, are larger), it may remain small for an extended period of time, 
with the other coordinates perhaps increasing.  Nonetheless, as we will see, 
after a finite expected time, one of the coordinates $k$, $k=2,\ldots,5$,
will hit $0$.  Because of the reflections against 0 by this coordinate and perhaps by the other three 
coordinates, the coordinate $k=1$ will acquire, on the average, a positive drift and therefore increase linearly.  
When this occurs, each of the coordinates $k=2,\ldots,5$ will drift towards $0$ and afterwards remain close to $0$.

The sixth coordinate increases linearly in time
when the first coordinate undergoes repeated reflection.  However, when the first coordinate is instead increasing,
the sixth coordinate will drift back to $0$ on account of the terms $(J_2)_{6,j} = -\delta_3$, $j=2,\ldots,5$.  
Moreover, on account of (\ref{eq1.6.3}), the term $(J_2)_{1,6} = -\delta_4$ is sufficiently smaller than $-\delta_2$
so that, when the sixth coordinate 
starts reflecting at $0$, the negative drift induced in the first coordinate more than compensates for 
the positive drift induced in the first coordinate by the reflection of the other four coordinates.
As a consequence, the first coordinate acquires a negative net drift.  After this occurs,
the coordinates $k=2,\ldots,6$ will all remain close to $0$ until the first coordinate hits $0$, in which case 
the behavior outlined above can repeat.  This behavior prevents any of the coordinates from typically 
moving too far from $0$, and will ensure that the system is positive recurrent.   

The proof of Theorem \ref{thm1.7.2} is organized as follows.  In Section 2, we give a number of bounds on $Y(\cdot)$
and $Z(\cdot)$ that are derived by applying elementary Brownian motion estimates to (\ref{eq1.2.1}).  These bounds are
employed in the rest of the paper. In Section 3, we demonstrate a version of Foster's criterion that will be used here.
We also recall and then employ the main result in Ratzkin and Treibergs \cite{RT}, which states that for a Brownian
pursuit problem, the presence of four ``predators" is enough for them to capture the ``prey" in finite expected time.
In our context, $Z_{k}(\cdot), k=2,\ldots,5$, will play the role of the predators and $Z_{1}(\cdot)$ will play 
the role of the prey.
This behavior will justify the claim in the above discussion that one of the coordinates with $k=2,\ldots,5$ will hit $0$
after a finite expected time.  

In Section 4, we state the main steps in the proof of Theorem \ref{thm1.7.2} in the form of a series
of five propositions, and show how the theorem follows
from them.  Depending on whether or not $Y_1(\cdot)$ is initially growing quickly, 
Proposition \ref{prop4.2.1} states that, during this time, either the coordinates 
$Z_2(\cdot),\ldots,Z_6(\cdot)$ decrease by an appropriate factor or 
$Z_6(\cdot)$ increases linearly.  In the first case, it follows from Proposition \ref{prop4.5.1} that
$Z_1(\cdot)$ will also remain small and so, as desired, the norm of the SRBM decreases by a factor over 
the time interval.  In the second case, the argument proceeds along the lines sketched above in the comparison of $Z(\cdot)$
with the divergent fluid path, and employs Propositions \ref{prop4.6.1}, \ref{prop4.21.1} and \ref{prop4.8.1}.  

In Section 5, we demonstrate Propositions \ref{prop4.2.1} and \ref{prop4.5.1} and, in Section 6, we demonstrate
Propositions \ref{prop4.6.1}, \ref{prop4.21.1} and \ref{prop4.8.1}.  The reasoning employs the interaction of the 
different components $Z_{k}(\cdot)$, $k=1,\ldots,6$,  and draws from the different bounds in Sections 2 and 3.  
\section{Basic estimates}
In this section, we give a number of elementary bounds that will be used in the remainder of the article.  In Lemma
\ref{lem2.1.1}, we give bounds on standard 
one dimensional Brownian motion $B(\cdot)$.  
(All of the bounds in the lemma hold in greater generality;
see, e.g., \cite{LT}, page 59, and \cite{L}.)
These bounds will then be applied in the
rest of the section to obtain bounds on the quantities $Y(\cdot)$ and $Z(\cdot)$ in (\ref{eq1.2.1}), the
equation describing the evolution of SRBM.   
Here and elsewhere in the paper, the notation $C_1,C_2,\ldots$ will be employed
for positive constants whose precise value is not of interest to us, with the same symbol often being reused.
\begin{lem}
\label{lem2.1.1}
Let $B(\cdot)$ denote a standard Brownian motion.  Then, for each $t\ge0$,
\begin{equation}
\label{eq2.1.2}
E\left[ \max_{0\le s\le s^{\prime}\le t} (B(s^{\prime}) - B(s))^2\right] \le 8t.
\end{equation}
For given $\epsilon > 0$, there exist $\aaa$, $\epsilon^{\prime}>0$ such that, for each $t\ge0$, 
\begin{equation}
\label{eq2.1.2'}
P\left(\max_{0\le s\le t} |B(s)| \ge \epsilon t \right) \le \aaa e^{-\epsilon^{\prime}t}.
\end{equation}
For given $\epsilon > 0$, there exist $\aaa$, $\epsilon^{\prime} > 0$ such that, for each $u\ge0$,
\begin{equation}
\label{eq2.1.3}
P\left(\inf_{t\ge 0} (\epsilon t + u - |B(t)|) \le 0 \right) \le \aaa e^{-\epsilon^{\prime}u},
\end{equation}
and, for each $u>0$ and $t\ge 0$,
\begin{equation}
\label{eq2.1.4}
P\left( \min_{0\le s\le s^{\prime}\le t} (\epsilon (s^{\prime} -s) + u - |B(s^{\prime}) - B(s)|)\le 0\right) \le 
\aaa (t+1)e^{-\epsilon^{\prime}u}.
\end{equation}
\end{lem}
\begin{proof}[Proof]
Since $(B(s^{\prime}) - B(s))^2 \le 2(B(s^{\prime})^2 + B(s)^2)$, it follows from the Reflection Principle that the
left side of (\ref{eq2.1.2}) is at most
$$4E\left[ \max_{0\le s\le t} B(s)^2\right] \le 8 E\left[B(t)^2\right] = 8t.$$
The bound (\ref{eq2.1.2'}) follows by applying the Reflection Principle to both $B(\cdot)$ and $-B(\cdot)$.  

Again
applying the Reflection Principle to $B(\cdot)$ and $-B(\cdot)$, it follows that, for given $\epsilon>0$,
\begin{equation*}
\begin{split}
& P\left( \tfrac{1}{2}(\epsilon t^{\prime} + u) - \max_{0\le s \le t^{\prime}}|B(s)|  \le 0 \right)
\le 4P\left( \tfrac{1}{2}(\epsilon t^{\prime} + u) - |B(t^{\prime})|  \le 0 \right) \\
& \qquad \qquad \qquad \qquad \qquad \qquad \qquad \le \bbb \text{exp}(-(\epsilon t^{\prime} + u)^2/8t^{\prime}) 
\le \bbb e^{-\frac{1}{8}\epsilon u},
\end{split}
\end{equation*}
where $\bbb$ does not depend on $t^{\prime}$ or $u$.  Setting $t^{\prime} = 2^i$,
$i=0,1,2,\ldots$, one obtains bounds whose exceptional probabilities sum to at most 
$\aaa e^{-\epsilon^{\prime}u}$, for $\epsilon^{\prime} = \tfrac{1}{8}\epsilon$ and 
appropriate $\aaa$.  The bound in (\ref{eq2.1.3}) follows quickly from this.

It follows from (\ref{eq2.1.3}) that, for each $i=0,1,2,\ldots$,
\begin{equation}
\label{eq2.2.3}
P\left( \inf_{s^{\prime}\ge i} (\epsilon (s^{\prime} -i) +\tfrac{1}{2}u - |B(s^{\prime}) - B(i)|)\le 0\right) \le 
\aaa e^{-\frac{1}{2}\epsilon^{\prime}u}.
\end{equation}
Using the Reflection Principle, it is easy to check that, for appropriate 
$\ccc$, $\epsilon^{\prime \prime}>0$ and all $u\ge0$,
$$P\left( \max_{0\le s \le 1} |B(i+s) - B(i)|\ge \tfrac{1}{2}u -\epsilon \right)  \le 
\ccc e^{-\epsilon^{\prime \prime}u}.$$
Together with (\ref{eq2.2.3}), this implies
$$P\left( \inf_{s\in [i,i+1), s^{\prime}\ge s} (\epsilon (s^{\prime} -s) +u - |B(s^{\prime}) - B(s)|)\le 0\right) \le 
\aaa e^{-\epsilon^{\prime}u}$$
for new choices of $\aaa$ and $\epsilon^{\prime}$.  Summing over $i < t$ gives the bounds in (\ref{eq2.1.4}).
\end{proof}
The next lemma provides elementary upper and lower bounds on $Y_{k}(\cdot)$.
\begin{lem}
\label{lem2.3.2}
For each $t\ge 0$ and $\ell = 2,\ldots,5$,
\begin{equation}
\label{eq2.3.3}
\sum_{k=1}^{6} Y_{k}(t) \ge t - Z_{\ell}(0) - B_{\ell}(t)
\end{equation}
and, for each $\ell = 1,\ldots,6$,
\begin{equation}
\label{eq2.3.4}
\sum_{k=1}^{6} Y_{k}(t) \ge \tfrac{1}{2}(t -  Z_{\ell}(0) - B_{\ell}(t)).
\end{equation}
For each $t\ge0$ and $k= 1,\ldots, 6$,
\begin{equation}
\label{eq2.3.4'}
Y_{k}(t) \le t + \max_{0\le s\le t} (-B_{k}(s))
\end{equation}
and, for a given $\epsilon \ge 0$, there exist $\aaa$ and $\epsilon^{\prime} > 0$ so that
\begin{equation}
\label{eq2.3.5}
P(Y_{k}(t) \ge (1 + \epsilon)t) \le \aaa e^{-\epsilon^{\prime}t}.
\end{equation}
\end{lem}
\begin{proof}[Proof]
Since $\delta_3 \ge 0$, it follows from (\ref{eq1.2.1}) that, for $\ell = 2,\ldots,5$,
\begin{equation}
\label{eq2.3.6}
Z_{\ell}(t) \le Z_{\ell}(0) + B_{\ell}(t) -t + \sum_{k=1}^{6} Y_{k}(t),
\end{equation}
from which (\ref{eq2.3.3}) immediately follows.  Since each of the entries of $J_2$ in (\ref{eq1.6.2}) is less
than $1$, the analog of (\ref{eq2.3.6}) holds for $\ell = 1,\ldots,6$, but with the term $2\sum_{k=1}^{6} Y_{k}(t)$.
This implies (\ref{eq2.3.4}).

Let $\tau$ denote the time in $[0,t]$ at which $Y_{k}(t)$ is first attained, for given $k$.  
It follows from (\ref{eq1.2.1}) that
$$Y_{k}(t) \le \tau - B_{k}(\tau) \le t + \max_{0\le s\le t} (-B_{k}(s)),$$
which imples (\ref{eq2.3.4'}).  The bound (\ref{eq2.3.5}) follows from (\ref{eq2.3.4'}) and (\ref{eq2.1.2'}).
\end{proof}

We next obtain a number of upper bounds on $Z_{k}(\cdot)$.  The following lemma is elementary.
\begin{lem}
\label{lem2.4.2}
Let $B(\cdot)$ denote a standard Brownian motion.  For each $k$, $t$ and $x$,
\begin{equation}
\label{eq2.4.3}
P\left( \max_{0\le s\le t} Z_{k}(s) - Z_{k}(0) \ge 7t + x \right) \le 16 P(B(t) \ge x).
\end{equation}
Consequently, for all $t$, and appropriate $\aaa$ and $\epsilon^{\prime} > 0$,
\begin{equation}
\label{eq2.4.4}
P\left( \max_{0\le s\le t} Z_{k}(s) - Z_{k}(0) \ge 8t \right) \le \aaa e^{-\epsilon^{\prime}t}. 
\end{equation}
\end{lem}
\begin{proof}[Proof]
It follows from (\ref{eq1.2.1}) that, since all entries for $J_2$ in (\ref{eq1.6.2}) are at most $\frac{1}{6}$, 
$$\max_{0\le s\le t}Z_{k}(s) - Z_{k}(0) \le \max_{0\le s\le t} B_{k}(s) + \tfrac{7}{6}\sum_{\ell =1}^{6} Y_{\ell}(t).$$
By (\ref{eq2.3.4'}) of Lemma \ref{lem2.3.2}, this is at most
$$7t + \tfrac{7}{6}\sum_{\ell =1}^{6} \max_{0\le s\le t}(-B_{\ell}(s)) + \max_{0\le s\le t}B_{k}(s).$$
The inequality in (\ref{eq2.4.3}) follows from this and the Reflection Principle.  The inequality in (\ref{eq2.4.4}) is
an immediate consequence of (\ref{eq2.4.3}).
\end{proof}

The following lemma requires a bit more work.  Here, we employ the notation $N_k(t)$, $k=1,\ldots,6$, with 
$N_6(t) = Y_1(t)$ and $N_k(t) = 0$ for $k\neq 6$; $x_+$ denotes the positive part of $x\in \mathbb{R}$.
\begin{lem}
\label{lem2.5.2}
For each $k$, $k=2,\ldots,6$, $t\ge 0$ and $x$,
\begin{equation}
\label{eq2.5.3}
P\left( \max_{0\le s\le t} Z_{k}(s) - Z_{k}(0) - \delta_1 N_k(t) \ge x \right) \le 24 P(B(t) \ge \tfrac{1}{4}x),
\end{equation}
where $B(\cdot)$ is standard Brownian motion.  Consequently, for given $\epsilon > 0$, there exist $\aaa$, 
$\epsilon^{\prime} > 0$ such that, for each $t\ge 0$,
\begin{equation}
\label{eq2.5.4}
P\left( \max_{0\le s\le t} Z_{k}(s) - Z_{k}(0) - \delta_1 N_k(t) \ge \epsilon t \right) \le \aaa e^{-\epsilon^{\prime}t}. 
\end{equation}
Also, for $k=2,\ldots,5$,
\begin{equation}
\label{eq2.6.1}
E\left[\left(\left(\max_{0\le s\le t}Z_k(t) - Z_k(0)\right)_+\right)^2\right] \le 24\cdot 16t.
\end{equation}
\end{lem}
\begin{proof}[Proof]
Let $\tau_k$ denote the last time $r$, $r\le s$, at which $Z_k(r) = 0$; if the set is empty, let $\tau_k = 0$.  
Let $\tau$ denote the last time $r$, $r\le s$, at which $Z_{\ell}(r) = 0$ for any $\ell = 2,\ldots,6$; denote
this coordinate by $L$.  If the set is empty, set $\tau = 0$.  We also abbreviate by setting $B_k(r_1,r_2) =
B_k(r_2) - B_k(r_1)$ and $N_k(r_1,r_2) = N_k(r_2) - N_k(r_1)$.

We claim that for given $k$, $k=2,\ldots,6$,   
\begin{equation}
\label{eq2.6.2}
Z_k(\tau) - Z_k(0) \le B_k(\tau_k,\tau)  -  B_{L}(\tau_k,\tau)  +\delta_1 N_k(\tau).
\end{equation}
To see this, note that subtraction of the equations for the $k^{\text{th}}$ and $L^{\text{th}}$ 
coordinates of (\ref{eq1.2.1}) implies
\begin{equation*}
\begin{split}
Z_k(\tau) - Z_k(\tau_k) & = Z_L(\tau) - Z_L(\tau_k) + B_k(\tau_k,\tau)  -  B_{L}(\tau_k,\tau) \\ 
& \quad+\delta_1 N_k(\tau_k,\tau) - \delta_1 N_L(\tau_k,\tau) - \delta_3 (Y_L(\tau) - Y_L(\tau_k)) \\
& \le B_k(\tau_k,\tau)  -  B_{L}(\tau_k,\tau)  +\delta_1 N_k(\tau).
\end{split}
\end{equation*}
When $\tau_k > 0$, $Z_k(\tau_k) = 0$ holds, and so (\ref{eq2.6.2}) follows.

Let $\tau^{\prime} = \tau \vee \tau_1$.  Note that, when $\tau \neq \tau^{\prime} > 0$, 
$$ Z_1(\tau^{\prime}) - Z_1(\tau) = -Z_1(\tau) \le 0. $$
Also, since $Z_6(r) > 0$ for $r \in (\tau, \tau^{\prime}]$, it follows from the definition of 
$J_2$ that  
$$ (R(Y(\tau^{\prime}) - Y(\tau)))_k \le (R(Y(\tau^{\prime}) - Y(\tau)))_1. $$
Subtraction of the $k^{\text{th}}$ and $1^{\text{st}}$ coordinates of (\ref{eq1.2.1}),
together with these two inequalities, implies that
\begin{equation}
\label{eq2.6.2'}
Z_k(\tau^{\prime}) - Z_k(\tau) \le B_k(\tau,\tau^{\prime})  -  B_{1}(\tau,\tau^{\prime})  +\delta_1 N_k(\tau,\tau^{\prime}).
\end{equation}

It is easy to see that
\begin{equation}
\label{eq2.6.3}
Z_k(s) - Z_k(\tau^{\prime}) \le B_k(\tau^{\prime},s).
\end{equation}
Combining (\ref{eq2.6.2}), (\ref{eq2.6.2'}) and (\ref{eq2.6.3}) implies
\begin{equation}
\label{eq2.6.4}
Z_k(s) - Z_k(0) \le B_k(\tau_k,s)  - B_{L}(\tau_k,\tau) - B_{1}(\tau,\tau^{\prime}) +\delta_1 N_k(t).
\end{equation}
One therefore obtains, for all $x$, that
\begin{equation}
\label{eq2.6.6}
\begin{split}
&  P\left( \max_{0\le s\le t} Z_{k}(s) - Z_{k}(0)   - \delta_1 N_k(t) \ge x \right) \\ 
& \qquad \qquad \qquad \le 6 P\left(\max_{0\le s\le s^{\prime}\le t}(B(s^{\prime}) - B(s)) \ge \tfrac{1}{2}x \right).
\end{split}
\end{equation}
%
%
%
%
%
%
%
It follows from the Reflection Principle that the right side of (\ref{eq2.6.6}) is at most $24P(B(t)\ge \frac{1}{4}x)$,
which implies (\ref{eq2.5.3}).

The inequalities in (\ref{eq2.5.4}) and (\ref{eq2.6.1}) follow directly from (\ref{eq2.5.3}).
\end{proof}

We will employ (\ref{eq2.5.3}) to show (\ref{eq2.7.2}) of the following lemma.  On account of the thin tail of
$\max_{0\le s\le t} Z_k(s)$, restricting its expectation to a set $F$ decreases the expectation proportionally to $P(F)$, 
except for a logarithmic factor; a similar statement holds for the second moment.  
The lemma will be important for our calculations later in the article. 
\begin{lem}
\label{lem2.7.1}
For an appropriate constant $\aaa$, all $t\ge 0$ and all measurable sets $F$ with $P(F)>0$,
\begin{equation}
\label{eq2.7.2}
E\left[ \max_{0\le s\le t} Z_{k}(s)^2; \,F \right] \le \aaa P(F)(t\log(e/P(F)) + M) 
\end{equation}
for $k=2,\ldots,5$, when $Z_k(0) \le \sqrt{M}$, and 
\begin{equation}
\label{eq2.7.3}
E\left[ \max_{0\le s\le t} Z_{k}(s); \,F \right] \le \aaa P(F)(\sqrt{t}\log(e/P(F)) + t + M) 
\end{equation}
for $k=1,2,\ldots,6$, when $Z_k(0) \le M$.
\end{lem}
\begin{proof}[Proof]
On account of (\ref{eq2.5.3}), we can construct a standard normal random variable $W$ on 
the probability space so that, for $k=2,\ldots,5$,
\begin{equation}
\label{eq2.8.1}
E\left[ \max_{0\le s\le t} ((Z_{k}(s) - Z_{k}(0))_+)^2; \,F \right] \le \bbb t E[W^2; \,F],
\end{equation}
where $\bbb = 24\cdot 16$.  (The inequality follows by integrating by parts and employing 
$E[(4B(t))^2] = 16t$.) 
Choosing $a$ so that $P(W^2 \ge a) = P(F)$, the right side of (\ref{eq2.8.1}) is at most
\begin{equation}
\label{eq2.8.2}
\bbb t \left(aP(F) + \int_{a}^{\infty} P(W^2\ge x)\,dx \right).
\end{equation}
The random variable $W^2$ has an exponentially tight tail in the sense that, for appropriate $\ccc,\ddd >0$
and all $y$, $x$ with $0\le y\le x$,
\begin{equation}
\label{eq2.8.3}
P(W^2 \ge x) \le \ccc e^{-\ddd (x-y)} P(W^2 \ge y).
\end{equation}
Setting $y=0$ and $x=a$, this implies $a\le \tfrac{1}{\ddd}\text{log}(\ccc/P(F))$.  Application of 
(\ref{eq2.8.3}) with $y=a$ therefore implies (\ref{eq2.8.2}) is at most 
\begin{equation*}
\bbb t(aP(F) + (\ccc / \ddd)P(W^2\ge a)) \le (\bbb /\ddd) t P(F)(\text{log}(1/P(F)) + \ccc + \text{log} \,\ccc).
\end{equation*}
So, for appropriate $\eee$,
\begin{equation}
\label{eq2.8.4}
E\left[ \max_{0\le s\le t} ((Z_{k}(s) - Z_k(0))_+)^2; \,F \right] \le \eee tP(F)\text{log}(e/P(F)). 
\end{equation}
For $Z_k(0) \le \sqrt{M}$, (\ref{eq2.7.2}) follows from this by considering the complementary events 
$\{\max_{s\le t} Z_k(s) > 2\sqrt{M}\}$ and $\{\max_{s\le t} Z_k(s) \le 2\sqrt{M}\}$, and noting that, on the former,
$$\max_{0\le s\le t} Z_{k}(s)^2 \le 4\max_{0\le s\le t} (Z_{k}(s) - Z_k(0))^2$$
and, on the latter, $\max_{s\le t} Z_k(s)^2 \le 4M$.

In order to show (\ref{eq2.7.3}), we note that, for $k=1,\ldots,6$, it follows from (\ref{eq1.2.1}) and (\ref{eq2.3.4'})
that
\begin{equation}
\label{eq2.9.1}
\begin{split}
Z_k(s) - Z_k(0) & \le B_k(s) - s + \tfrac{7}{6}\sum_{\ell = 1}^{6} Y_{\ell}(s) \\
& \le 6s + B_k(s) + \tfrac{7}{6}\sum_{\ell = 1}^{6}\max_{0\le r\le s}(-B_{\ell}(r)).
\end{split}
\end{equation}
This, together with the Reflection Principle, implies that
\begin{equation}
\label{eq2.9.2}
\begin{split}
P\left(\max_{0\le s\le t} Z_k(s) - Z_k(0) - 6t \ge x\right) &\le 7P\left(\max_{0\le s\le t} B(s) \ge \tfrac{1}{8}x\right)\\
&\le 14P(B(t) \ge \tfrac{1}{8}x),
\end{split}
\end{equation}
where $B(\cdot)$ is standard Brownian motion.

Reasoning as in the first part of the proof, we can construct a standard normal random variable $W$ so that
\begin{equation}
\label{eq2.9.3}
E\left[\max_{0\le s\le t} Z_k(s) - Z_k(0) - 6t;\, F\right] \le \bbb \sqrt{t} E[W;\, F],
\end{equation}
where $\bbb = 14 \cdot 8$.  Since $W$ has an exponentially tight tail, we can reason as through (\ref{eq2.8.4})
to show that 
\begin{equation}
\label{eq2.10.1}
E\left[\max_{0\le s\le t} Z_k(s) - Z_k(0) - 6t; \, F \right] \le \eee \sqrt{t}P(F) \text{log}(e/P(F))
\end{equation}
for appropriate $\eee$.  This implies (\ref{eq2.7.3}) for $Z_k(0) \le M$ and appropriate $\aaa$.
\end{proof}

We now apply Lemma \ref{lem2.5.2} to obtain sharper bounds on $Y_k(\cdot)$, with $k=2,\ldots,6$, than those in 
Lemma \ref{lem2.3.2}, provided bounds on $Y_1(\cdot)$ are given.
\begin{lem}
\label{lem2.10.2}
For given $\epsilon > 0$, there exist $\aaa$ and $\epsilon^{\prime} > 0$ such that, for all $t\ge 0$ and
$k=2,\ldots,6$,
\begin{equation}
\label{eq2.10.4}
P\left(Y_k(t) + (1-\delta_3)\sum_{\ell=2,\,\ell\neq k}^{6}Y_{\ell}(t) \ge (1 + \epsilon)t + \delta_1 N_k(t)\right) 
\le \aaa e^{-\epsilon^{\prime}t}.
\end{equation}
There exist $\aaa$ and $\epsilon^{\prime} > 0$ such that, for all $t\ge 0$ and $k=2,\ldots,6$,
\begin{equation}
\label{eq2.10.3}
P\left(Y_k(t) \le \tfrac{1}{5}t - \tfrac{1}{\delta_3} (Z_k(0) + 2Y_1(t))\right) \le \aaa e^{-\epsilon^{\prime}t}.
\end{equation}
\end{lem}
\begin{proof}[Proof]
It follows from (\ref{eq1.2.1}) that
\begin{equation}
\label{eq2.10.5}
Z_k(t) - Z_k(0) \ge B_k(t) - t + Y_k(t) +(1 - \delta_3) \sum_{\ell=2,\,\ell\ne k}^{6} Y_{\ell}(t)
\end{equation}
for $k=2,\ldots,6$.  Together with (\ref{eq2.5.4}) of Lemma \ref{lem2.5.2}, (\ref{eq2.10.5}) implies (\ref{eq2.10.4}).

Summing the arguments inside $P(\cdot)$ in (\ref{eq2.10.4}), over $\ell=2,\ldots,6$, gives
\begin{equation*}
P\left((1-\tfrac{4}{5}\delta_3) \sum_{\ell=2}^{6} Y_{\ell}(t) \ge (1+\epsilon)t + \delta_1 Y_1(t) \right)
\le 5\aaa e^{-\epsilon^{\prime}t}.
\end{equation*}
Since $\delta_3 \le \tfrac{1}{10}$,
\begin{equation}
\label{eq2.11.1}
\frac{1-\delta_3}{1-\frac{4}{5}\delta_3} \le 1 - \tfrac{1}{5}\delta_3 - \tfrac{1}{10}\delta_{3}^{2},
\end{equation}
which implies that, for small enough $\epsilon$,
\begin{equation}
\label{eq2.11.2}
P\left((1-\delta_3) \sum_{\ell=2}^{6} Y_{\ell}(t) -t \ge -
\tfrac{1}{5}(1+\tfrac{1}{2}\delta_3)\delta_3 t + \delta_1 Y_1(t) \right)
\le 5\aaa e^{-\epsilon^{\prime}t}.
\end{equation}
By (\ref{eq1.2.1}), one has, for $k=2,\ldots,6$,
$$Z_k(t) - Z_k(0) \le B_k(t) - t + \delta_3 Y_k(t) +(1 - \delta_3) \sum_{\ell=2}^{6} Y_{\ell}(t) + (1+\delta_1)Y_1(t).$$
Off of the exceptional set in (\ref{eq2.11.2}), this is at most
$$ -\tfrac{1}{5}(1+\tfrac{1}{2}\delta_3)\delta_3 t +B_k(t) + \delta_3 Y_k(t) + 2Y_1(t).$$
Solving for $Y_k(t)$, together with the obvious exponential bound on $B_k(t)$, produces (\ref{eq2.10.3}) for a new choice
of $\aaa$ and $\epsilon^{\prime}$.
\end{proof}

In Lemma \ref{lem2.5.2}, we gave upper bounds on $Z_k(\cdot)$ for $k=2,\ldots,6$.  Here, we employ (\ref{eq2.10.4})
and (\ref{eq2.10.3}) of Lemma \ref{lem2.10.2} to obtain an upper bound on $Z_1(\cdot)$.  The bound implies in
particular that, for large $t$, $Y_1(t) > 0$ and hence $Z_1(s) = 0$ at some $s\le t$.
\begin{lem}
\label{lem2.12.1}
For given $\epsilon > 0$, there exist $\aaa$ and $\epsilon^{\prime} > 0$ such that, for each $t\ge 0$,
\begin{equation}
\label{eq2.12.2}
P\left(Z_1(t) - Z_1(0) - \tfrac{3}{\delta_3}(Y_1(t) + Z_6(0)) \ge -\tfrac{1}{30}\delta_4 t \right) 
\le \aaa e^{-\epsilon^{\prime}t}.
\end{equation}
\end{lem}
\begin{proof}[Proof]
On account of (\ref{eq1.2.1}),
\begin{equation}
\label{eq2.12.3}
Z_1(t) - Z_1(0) \le B_1(t) - t + Y_1(t) + (1 + \delta_2) \sum_{k=2}^{6} Y_{k}(t) - (\delta_2 +\delta_4)Y_6(t).
\end{equation}
One bounds $(1+\delta_2)\sum_{k=2}^{6}Y_k(t)$ by employing (\ref{eq2.10.4}) after summing over $\ell =2,\ldots,6$,
and one bounds $(\delta_2 + \delta_4)Y_6(t)$ by employing (\ref{eq2.10.3}).  It then follows with a little
algebra that the right side of (\ref{eq2.12.3}) is at most
\begin {equation}
\label{eq2.12.4}
\begin{split}
(\tfrac{9}{10}\delta_2 + \delta_3 - \tfrac{1}{5}\delta_4)t + & \tfrac{3}{\delta_3} (Y_1(t) + Z_6(0)) \\
& \le  -\tfrac{1}{30}\delta_4 t + \tfrac{3}{\delta_3}(Y_1(t) + Z_6(0))
\end{split}
\end{equation}
off of a set of probability $\aaa e^{-\epsilon^{\prime}t}$, for appropriate $\aaa$ and $\epsilon^{\prime} > 0$.  For the
bound on the left side of (\ref{eq2.12.4}), one employs the bounds on $\delta_i$ 
in (\ref{eq1.6.3}) and (\ref{eq1.6.4}), together with (\ref{eq2.10.4}), (\ref{eq2.10.3}) 
and an analog of (\ref{eq2.11.1}).  For the
inequality in (\ref{eq2.12.4}), one uses $\delta_2 + \delta_3 \le \frac{1}{6}\delta_4$.  It follows from
(\ref{eq2.12.3}) and (\ref{eq2.12.4}) that, off of the exceptional set,
$$Z_1(t) - Z_1(0) \le -\tfrac{1}{30}\delta_4 t +\tfrac{3}{\delta_3}(Y_1(t) + Z_6(0)),$$
which implies (\ref{eq2.12.2}).
\end{proof}
\section{A Brownian pursuit model and Foster's criterion}
In this section, we first discuss a Brownian pursuit model, which was mentioned briefly at the end of 
Section 1.  Using a result of Ratzkin and Treibergs \cite{RT}, it is employed to show that the 
expected time for at least one of the coordinates $Z_k(\cdot)$, $k=2,\ldots,5$, to hit $0$ is finite.  In 
Proposition \ref{prop3.5.2}, we apply this result to obtain a lower bound on $\sum_{k=2}^{5} Y_k(\cdot)$ 
that will be used later in the
paper.  We then show an appropriate version of Foster's criterion.  Foster's criterion is a tool for showing
the positive recurrence of a Markov process.  Since the stopping times we will employ are random, we need a variant of the
standard version.

\subsection*{A Brownian pursuit model}
The pursuit model consists of $n$ standard $1$-dimensional Brownian motions, $X_{k}(\cdot)$, $k=2,\ldots,n+1$, 
that ``pursue" another Brownian motion $X_{1}(\cdot)$.    
The $n$ Brownian motions are referred to as \emph{predators} and the other Brownian motion as the \emph{prey}.  The prey will be
said to be \emph{captured} at time $t$ if $t$ is the first time at which $X_{1}(t) = X_{k}(t)$ for some $k=2,\ldots,n+1$.  All 
Brownian motions are assumed to move independently.  

One wishes to know whether the expected time for capture is finite or infinite.  When there are initially predators on each
side of the prey, one can show that the expected capture time is finite.  When all of the predators are on one side of the
prey, the expected capture time is infinite for $n \le3$ and finite for $n\ge4$.  This and a number of related problems
were considered in Bramson and Griffeath \cite{BG} in the context of simple symmetric random walk.  There, the behavior 
for  $n \le3$ was demonstrated and simulations were given that suggested the behavior for $n\ge4$.  Li and Shao \cite{LS}
showed finite expected capture time for Brownian motion for $n\ge5$ and Ratzkin and Treibergs \cite{RT} 
more recently showed this for $n=4$.  

Ratzkin and Treibergs \cite{RT} showed finite expected capture time by bounding the tail of the capture time $T$.  
Their result can be formulated as follows:
\begin{thm}
\label{thm3.1.1}
For any initial state where all four of the predators are within distance $1$ and to the right of the prey, 
\begin{equation}
\label{eq3.1.2}
P(T>t) \le \aaa t^{-(1+\eta)}
\end{equation}
for appropriate $\aaa$ and all $t\ge0$, where $\eta = .000073$.  Consequently,
\begin{equation}
\label{eq3.1.3}
E[T] < \infty.
\end{equation}
\end{thm}
The analogous result for $n=5$ is less delicate, which \cite{RT} showed with $\eta = .0634$.  The reasoning in both
\cite{LS} and \cite{RT} relies on rephrasing the pursuit model in terms of an eigenvalue problem for the departure
time of an $n$-dimensional Brownian motion from an appropriate generalized cone. This type of problem was also studied in
DeBlassie \cite{DeB}.  (See \cite{LS} for additional references.)

We will employ both (\ref{eq3.1.2}) and (\ref{eq3.1.3}) for Proposition \ref{prop3.5.2}.  
(The first inequality is not needed, but applying it
makes one of the steps more explicit.)  We note that, by (\ref{eq1.2.1}), for $k=2,\ldots,5$ and all $t\ge0$,
\begin{equation}
\label{eq3.2.1}
\begin{split}
Z_k(t) - Z_1(t) \le & (Z_k(0) - Z_1(0)) + (B_k(t) - B_1(t)) \\ 
& -\delta_2 \sum_{k=2}^{5} Y_k(t) + (\delta_4 - \delta_3) Y_6(t).
\end{split}
\end{equation}
When $Y_6(t) = 0$, this implies
\begin{equation}
\label{eq3.2.2}
Z_k(t) - Z_1(t) \le (Z_k(0) - Z_1(0)) + (B_k(t) - B_1(t)).
\end{equation}
Set
\begin{align}
\label{eq3.3.0a}
& T_1(x) = \min\{t: Z_1(t) - Z_1(0) \ge x\}, \\
\label{eq3.3.0b}
& T_6 = \min\{t: Z_6(t) = 0\}.
\end{align} 
By employing Theorem \ref{thm3.1.1} and (\ref{eq3.2.2}), it is easy to show the following proposition.
\begin{prop}
\label{prop3.3.1}
Suppose that for a given $x\ge0$, $max_{k=2,\ldots,5}Z_k(0) \le x$.  Then, for $\eta = .000073$ and
an appropriate constant $\aaa$ not depending on $x$,
\begin{equation}
\label{eq3.3.2}
P(T_1(x) \wedge T_6 \ge x^2 t) \le \aaa t^{-(1+\eta)}
\end{equation}
for all $t\ge0$.  Consequently, for appropriate $\bbb$ not depending on $x$,
\begin{equation}
\label{eq3.3.3}
E[T_1(x) \wedge T_6 ] < \bbb x^2.
\end{equation}
\end{prop}
\begin{proof}
By scaling space and time by $2x$ and $4x^2$, respectively, it follows from (\ref{eq3.1.2}) of Theorem \ref{thm3.1.1} that
$$ P\left(B_1(s) - \min_{2\le k\le5} B_k(s) < 2x \text{ for all } s\le x^2 t\right) \le \aaa t^{-(1+\eta)}$$
for a new choice of $\aaa$.  On account of (\ref{eq3.2.2}) and the bounds on $Z_k(0)$, $k=2,\ldots,5$, 
this implies that
\begin{equation*}
\begin{split}
P(Z_1 &  (s) - Z_1(0) < x \text{ for all } s\le x^2 t; \,T_6\ge x^2 t) \\ 
& \le P\left(Z_1(s) - Z_1(0) - \min_{2\le k\le 5} Z_k(s) < x \text{ for all } s\le x^2 t;\, T_6  \ge x^2 t\right) \\ 
& \le \aaa t^{-(1+\eta)}.
\end{split}
\end{equation*} 
The inequality in (\ref{eq3.3.2}) follows immediately.  
\end{proof}

\subsection*{Application of Proposition \ref{prop3.3.1}}
We define the stopping times
%
%
\begin{equation}
\label{eq3.5.1}
T_2(x) = \min\left\{t: \sum_{k=2}^{5} Y_k(t) = \tfrac{1}{2}(t+x^2)\right\} \wedge T_6 \wedge 5x^{5/\eta},
\end{equation}
%
%
where $\eta$ is as in Theorem \ref{thm3.1.1}.  
%
%
In Sections 4-6, we will require upper bounds on $E[T_2(x)]$ in order to ensure the linear growth  of $Z_1(\cdot)$
mentioned at the end of Section 1.  Here, we employ Proposition \ref{prop3.3.1} to obtain the following bounds.
\begin{prop}
\label{prop3.5.2}
Suppose that $\max_{k=2,\ldots,5} Z_k(0) \le x$, with $x\ge 2$.  Then, for appropriate $\aaa$ not depending on
$x$, 
\begin{equation}
\label{eq3.5.3}
E[T_2(x)] \le \aaa x^2.
\end{equation}
\end{prop}

In Sections 4-6, we will also require upper bounds on $P(A)$, where
\begin{equation}
\label{eq3.20.1}
A = \{\omega: T_2(x) = 5 x^{5/\eta} \}.
\end{equation}
These bounds are obtained in Proposition \ref{prop3.7.1}, which we state shortly.

In order to demonstrate Propositions \ref{prop3.5.2} and \ref{prop3.7.1}, we need to rule out 
certain behavior of $Z(\cdot)$ except on sets
of small probability.  For this, we introduce the following notation.  Let $S_1(x)$ denote the last time $t$ before 
$T_1(x)$ at which $Z_1(t) - Z_1(0) = \tfrac{1}{2}x$, for given $x$.  Set
\begin{equation}
\label{eq3.6.1}
\tau = \min \left\{t: \min_{2\le k\le 5} Z_k(t) = 0 \text{ for } t\ge S_1(x) \right\}.
\end{equation}
(If $\tau$ does not occur, set $\tau= \infty$.)  Neither $S_1(x)$ nor $\tau$ is a stopping time.  We also set
\begin{equation}
\label{eq3.6.2}
t_e = x^{5/\eta}, \quad t_f = 5x^{5/\eta} \quad \text{ and } \quad T_{1}^{\prime} (x) = 4(T_1(x) + x^2).
\end{equation}

Using this notation, we define:
\begin{align}
\label{eq3.6.3}
& A_1 = \{\omega: T_1(x) \wedge T_6 > t_e \}, \\
\label{eq3.6.4}
& A_2 = \{\omega: T_1(x) \le T_6 \wedge \tau \wedge t_e \}, \\
\label{eq3.6.5}
& A_3 = \{\omega: \tau < T_1(x) \le t_e,\, T_6 > T_{1}^{\prime}(x) \}, \\
\label{eq3.6.7}
& A_4 = \left\{\omega: \sum_{k=2}^{5}Y_k(T_{1}^{\prime}(x))  < \tfrac{1}{2}(T_{1}^{\prime}(x)+x^2) \right\}, \\
\label{eq3.20.2}
& A_5 = \{\omega: T_6 > T_{1}^{\prime}(x)\}.
\end{align}
One can check that 
\begin{equation}
\label{eq3.20.3}
A_5 \subseteq A_1 \cup A_2 \cup A_3.
\end{equation}
Also, note that
\begin{equation}
\label{eq3.20.4}
T_2(x) \le T_{1}^{\prime}(x) \wedge T_6 \qquad \text{on } A_{4}^{c}.
\end{equation}

Using this notation, it is not difficult to show the following lemma.
\begin{lem}
\label{lem3.20.5}
For $A$ as in (\ref{eq3.20.1}),
\begin{equation}
\label{eq3.20.6}
A \subseteq A^{\prime} \stackrel{\mathrm{def}}{=} A_1 \cup A_2 \cup (A_3 \cap A_4).
\end{equation}
\end{lem}
\begin{proof}
Set
$$ T_{2}^{\prime}(x) = \min\left\{t: \sum_{k=2}^{5} Y_k(t) = \tfrac{1}{2}(t+x^2)\right\} \wedge T_6 $$
and $A_6 = \{\omega: T_{2}^{\prime}(x) \ge t_f \}$.  It suffices to show that $A_6 \subseteq A^{\prime}$.

Since $T_{1}^{\prime}(x) < t_f$ on $A_3$,
$$ A_3 \cap A_6 \subseteq A_3 \cap A_4. $$
Consequently, by (\ref{eq3.20.3}) and the definition of $A^{\prime}$,
\begin{equation}
\label{eq3.20.7}
A_5 \cap A_6 \subseteq A_1 \cup A_2 \cup (A_3 \cap A_6) \subseteq A^{\prime}.
\end{equation}
On the other hand,
$$ A_{5}^{c}  \cap A_6 \subseteq A_1 \subseteq A^{\prime}. $$
Together with (\ref{eq3.20.7}), this implies $A_6 \subseteq A^{\prime}$, as desired.
\end{proof}

The bounds on $P(A^{\prime})$ in Proposition \ref{prop3.7.1} will be applied in the proof of Proposition \ref{prop3.5.2}
and the bounds on $P(A)$ will be applied in the proof of Proposition \ref{prop4.21.1}.  
Proposition \ref{prop3.3.1}, Lemma \ref{lem2.1.1} and (\ref{eq1.2.1}) are the main tools 
in the proof of Proposition \ref{prop3.7.1}.
\begin{prop}
\label{prop3.7.1}
Suppose that $\max_{k=2,\ldots,5} Z_k(0) \le x$, with $x\ge1$.  Then, for an appropriate $\aaa$
not depending on $x$,
\begin{equation}
\label{eq3.7.2}
P(A) \le P(A^{\prime}) \le \aaa x^{-\frac{5}{\eta} -2}.
\end{equation}
\end{prop}
\begin{proof}
In addition to $A_1$, $A_2$, $A_3$ and $A_4$, we employ the set
\begin{equation}
\label{eq3.6.6}
A_7 = \{\omega: Z_1(s) = 0 \text{ for some } s\in [\tau,T_{1}^{\prime}(x)] \}.
\end{equation}
We proceed to obtain upper bounds on each of $P(A_1)$, $P(A_2)$, $P(A_3 \cap A_7)$ and $P(A_3 \cap A_4 \cap A_{7}^{c})$.
We first note that, by applying (\ref{eq3.3.2}) of Proposition \ref{prop3.3.1}, with $t=x^{\frac{5}{\eta} - 2}$,
\begin{equation}
\label{eq3.7.4}
P(A_1) \le \bbb x^{-\frac{5}{\eta} - 2}
\end{equation}
for appropriate $\bbb$.

In order to bound $P(A_2)$, we need to show that, over $[S_1(x),t_e]$, $T_1(x)$ typically will not occur before
$T_6 \wedge \tau$ occurs; on this set, $Z_1(\cdot)$ will drift toward $0$ and away from $x$.  First note,
by (\ref{eq1.2.1}), that when 
$T_1(x) \le T_6 \wedge \tau$,
$$ Z_1(T_1(x)) - Z_1(S_1(x)) = (B_1(T_1(x)) - B_1(S_1(x))) - (T_1(x) - S_1(x)).$$
It then follows from the definitions of $T_1(x)$ and $S_1(x)$ that 
\begin{equation}
\label{eq3.8.2}
B_1(T_1(x)) - B_1(S_1(x)) = T_1(x) - S_1(x) + \tfrac{1}{2}x.
\end{equation}
But, by (\ref{eq2.1.4}) of Lemma \ref{lem2.1.1}, the probability of (\ref{eq3.8.2}) occurring when $T_1(x) \le t_e$
is at most 
$$ \bbb (t_e +1)e^{-\epsilon^{\prime} x} \le \ccc e^{-\frac{1}{2}\epsilon^{\prime} x} $$
for appropriate $\bbb$, $\ccc$ and $\epsilon^{\prime} > 0$.  Consequently,
\begin{equation}
\label{eq3.8.3}
P(A_2) \le \ccc e^{-\frac{1}{2}\epsilon^{\prime} x}. 
\end{equation}

We next show that $P(A_3 \cap A_7)$ is small.  This event will typically not occur because the coordinates
$k=2,\ldots,5$ that are reflecting at $0$ after $\tau$ will impart a positive drift to $Z_1(\cdot)$.  Restricting
our attention to the event $A_3$, let $K$ be the index at which $Z_K(\tau) = 0$.  Also, let $\tau^{\prime}$ be any
random time with
\begin{equation}
\label{eq3.9.1}
\tau \le \tau^{\prime} \le T_{1}^{\prime}(x) \wedge \min\{s>\tau: Z_1(s) = 0\}.  
\end{equation}
Since $\tau^{\prime} \le T_{1}^{\prime}(x) < T_6$, it follows from (\ref{eq1.2.1}) that
\begin{equation}
\label{eq3.9.2}
\sum_{k=2}^{5} (Y_k(\tau^{\prime}) - Y_k(\tau)) \ge (\tau^{\prime} - \tau) - (B_K(\tau^{\prime}) -B_K(\tau)).
\end{equation}
Applying (\ref{eq1.2.1}) for the first coordinate and then substituting in (\ref{eq3.9.2}), one obtains
\begin{equation}
\label{eq3.9.3}
Z_1(\tau^{\prime}) - Z_1(\tau) \ge 
\tilde{B}(\tau^{\prime}) - \tilde{B}(\tau) + \delta_2(\tau^{\prime} - \tau),
\end{equation}
where $\tilde{B}(t) \stackrel{\text{def}}{=} B_1(t) - (1 + \delta_2)B_K(t)$. 

Again applying (\ref{eq2.1.4}) of Lemma \ref{lem2.1.1}, one has
\begin{equation}
\label{eq3.9.4}
\begin{split}
P\left(\tilde{B}(\tau^{\prime}) - \tilde{B}(\tau) + \delta_2 (\tau^{\prime} - \tau) \le - \tfrac{1}{4}x \right)
& \le \bbb (T_{1}^{\prime}(x) + 1)e^{-\epsilon^{\prime}x} \\ 
& \le \ccc e^{-\frac{1}{2}\epsilon^{\prime}x},
\end{split}
\end{equation}
with $T_1(x) \le t_e$ and the definitions of $T_{1}^{\prime}(x)$ and $t_e$ being used in the latter inequality.
Applying this to (\ref{eq3.9.3}), one obtains that, since $Z_1(\tau) \ge \tfrac{1}{2}x$, 
$$ P(Z_1(\tau^{\prime}) \le \tfrac{1}{4}x; A_3) \le \ccc e^{-\frac{1}{2}\epsilon^{\prime}x} $$
for $\tau^{\prime}$ as in (\ref{eq3.9.1}).  This implies that
\begin{equation}
\label{eq3.10.2}
P(A_3 \cap A_7) \le \ccc e^{-\frac{1}{2}\epsilon^{\prime}x}.
\end{equation}

We now show that
\begin{equation}
\label{eq3.10.3}
P(A_3 \cap A_4 \cap A_{7}^{c}) \le \bbb e^{-\frac{1}{4}\epsilon^{\prime}x}
\end{equation}
for an appropriate choice of $\bbb$.  On the set $A_3 \cap A_{7}^{c}$, it follows from (\ref{eq1.2.1}) that
\begin{equation}
\label{eq3.10.4}
\sum_{k=2}^{5} Y_k(T_{1}^{\prime}(x)) \ge (T_{1}^{\prime}(x) - \tau) - (B_K(T_{1}^{\prime}(x)) - B_K(\tau)), 
\end{equation}
where $K$ is the index at which $Z_K(\tau) = 0$.  Since $\tau < T_1(x)$, it follows from the definition of
$T_{1}^{\prime}(x)$ that the right side of (\ref{eq3.10.4}) is at least
\begin{equation*}
\tfrac{1}{2}(T_{1}^{\prime}(x) + x^2) + 
[\tfrac{1}{2}x^2 + \tfrac{1}{4}(T_{1}^{\prime}(x) - \tau) - (B_K(T_{1}^{\prime}(x)) - B_K(\tau))]. 
\end{equation*}
Again applying (\ref{eq2.1.4}), this is greater than $\tfrac{1}{2}(T_{1}^{\prime}(x) +x^2)$ off of a set of
probability $\bbb e^{-\epsilon^{\prime}x^2}$, for appropriate $\bbb$ and $\epsilon^{\prime} > 0$. 
Consequently,
\begin{equation}
\label{eq3.11.1}
\begin{split}
P(A_3 & \cap A_4 \cap A_{7}^{c}) \\ 
& = P\left(\sum_{k=2}^{5} Y_k(T_{1}^{\prime}(x)) 
< \tfrac{1}{2} (T_{1}^{\prime}(x) + x^2); A_3 \cap A_{7}^{c}\right) \\
& \le \bbb e^{-\epsilon^{\prime}x^2}.  
\end{split}
\end{equation}

One has
$$ A \subseteq A^{\prime} = A_1 \cup A_2 \cup (A_3 \cap A_4) \subseteq A_1 \cup A_2 \cup (A_3 \cap A_7) 
\cup (A_3 \cap A_4 \cap A_{7}^{c}). $$
Combining (\ref{eq3.7.4}), (\ref{eq3.8.3}), (\ref{eq3.10.2}) and (\ref{eq3.11.1}) therefore implies (\ref{eq3.7.2})
for an appropriate choice of $\bbb$.
\end{proof}

Using Proposition \ref{prop3.7.1}, the demonstration of Proposition \ref{prop3.5.2} is quick.

\begin{proof} [Proof of Proposition \ref{prop3.5.2}]
%
%
%
It follows from (\ref{eq3.20.3}) that

$$ A_{4}^{c} \cup A_{5}^{c} \supseteq (A_1 \cup A_2 \cup  (A_3 \cap A_4))^c = (A^{\prime})^c . $$
%
%
Because of (\ref{eq3.20.4}),
\begin{equation}
\label{eq3.11.3}
T_2(x) \le T_{1}^{\prime}(x) \wedge T_6
\end{equation}
on $A_{4}^{c}$.  On the other hand, (\ref{eq3.11.3}) holds trivially on $A_{5}^{c}$.
%
%
%
%
Along with (\ref{eq3.6.2}), this implies that
\begin{equation}
\label{eq3.11.5}
T_2(x) \le 4(T_1(x) + x^2) \wedge T_6 \le 4(T_1(x) \wedge T_6) + 4x^2
\end{equation}
on $(A^{\prime})^c$, and so, by (\ref{eq3.3.3}) of Proposition \ref{prop3.3.1},
\begin{equation}
\label{eq3.12.1}
E[T_2(x); (A^{\prime})^c] \le \ccc x^2
\end{equation}
for appropriate $\ccc$.

The bound $T_2(x) \le 5x^{5/\eta}$ always holds and so, by Proposition \ref{prop3.7.1},
\begin{equation}
\label{eq3.12.2}
E[T_2(x); A^{\prime}] \le \ddd/x^2
\end{equation}
for appropriate $\ddd$.  Inequality (\ref{eq3.5.3}) follows immediately from (\ref{eq3.12.1}) and (\ref{eq3.12.2}).
\end{proof}
\subsection*{Foster's criterion}
Foster's criterion is a standard tool for showing positive recurrence 
of a Markov process when the process has a ``uniformly negative drift" off of a bounded set in the state space 
(see, e.g., Bramson \cite{B} or Foss and Konstantopoulos \cite{FK} ).  Versions of Foster's criterion typically 
employ deterministic stopping times   
whose length depends only on the initial state. Here, we require a version of Foster's criterion with random times,
which is given below.

We state the proposition for SRBM defined on the induced $Z$-path space, consisting of continuous paths
on $\mathbb{R}_{+}^{6}$ with the natural filtration, in order to facilitate the definition
of the sequence of stopping times employed in its proof.  The SRBM can always be projected onto this space.  
The proof of the proposition employs an elementary martingale argument that extends to more general Feller 
processes.  

Here and later on in the article, we employ the norm
\begin{equation}
\label{eq3.13.1}
\|z\| = z_1 + \sum_{k=2}^{5} z_{k}^{2} + z_6 \qquad \text{for } z = (z_1,\ldots,z_6), z_k \ge 0.
\end{equation}
We set, for $\delta > 0$,
\begin{equation*}
\tau_A(\delta) = \inf\{t \ge \delta: Z(t) \in A\}; 
\end{equation*}
$E_z[\cdot]$ denotes the expectation for the process with $Z(0) = z$, and $\mathcal{F}(t), t\ge 0$, denotes the filtration
of $\sigma$-algebras associated with the SRBM.
%
%
\begin{prop}
\label{prop3.14.1}
Suppose that, for some $\delta,\epsilon,\kappa > 0$ and a family of stopping times $\sigma (z)$,
$z\in \mathbb{R}_{+}^{6}$, with $\sigma (z) \ge \delta$,  $E_z[\sigma (z)]$ is measurable in $z$ 
and the SRBM $Z(\cdot)$ satisfies
\begin{equation}
\label{eq3.14.2}
E_z [\|Z(\sigma (z))\|] \le (\|z\| \vee \kappa) - \epsilon E_z[\sigma (z)]
\end{equation}
%
%
for all $z$.  Then
\begin{equation}
\label{eq3.14.3}
E_z[\tau_A  (\delta)] \le \frac{1}{\epsilon}(\|z\| \vee \kappa) \qquad \text{for all } z,
\end{equation}
where $A = \{z: \|z\| \le \kappa\}$.  Hence, $Z(\cdot)$ is positive recurrent. 
\begin{proof} 
The argument is a slight modification of that for the generalized Foster's criterion given on page 94 of \cite {B}.
Set $\sigma_0=0$, and let $\sigma_1 < \sigma_2 < \ldots$ denote the stopping times defined inductively,
with $\sigma_n - \sigma_{n-1}$, conditioned on $Z(\sigma_{n-1}) = z$, having the 
same law as $\sigma(z)$ given $Z(0) = z$.  By
(\ref{eq3.14.2}) and the strong Markov property, for all $z$,
\begin{equation}
\label{eq3.14.3'}
E_z [\|Z(\sigma_n)\| \, | \, \mathcal{F}(\sigma_{n-1})] \le (\|Z(\sigma_{n-1})\| \vee \kappa) 
- \epsilon E_{Z(\sigma_{n-1})}[\sigma(Z(\sigma_{n-1}))]
\end{equation}
for almost all $\omega$.

Set $M(0) = \|z\| \vee \kappa$ and
\begin{equation}
\label{eq3.14.4}
M(n) = \|Z(\sigma_n)\| + \epsilon \sigma_n  \qquad \text{for } n\ge1.
\end{equation}
Also, set $\mathcal{G}(n) = \mathcal{F}(\sigma_n)$.  On account of (\ref{eq3.14.3'}),
\begin{equation}
\label{eq3.14.5}
E_z[M(n) \, | \, \mathcal{G}(n-1)] \le M(n-1) \qquad \text{for } n\le \rho,
\end{equation}
where $\rho$ is the first time $n > 0$ at which $M(n)\in A$.  So, $M(n \wedge \rho)$ is a nonnegative supermartingale
on $\mathcal{G}(n)$.

It follows from the Optional Sampling Theorem that
\begin{equation}
\label{eq3.15.1}
E_z[M(\rho)] \le \|z\| \vee \kappa.
\end{equation}
Note that $\tau_A (\delta) \le \sigma_{\rho}$.  Therefore, by (\ref{eq3.14.4}) and (\ref{eq3.15.1}),
\begin{equation}
\label{eq3.15.2}
\epsilon E_z [\tau_A (\delta)] \le E_z [M(\rho)] \le \|z\| \vee \kappa,
\end{equation}
which implies (\ref{eq3.14.3}) as desired.
\end{proof}

\end{prop}

\section{Main steps of the proof of Theorem \ref{thm1.7.2}}
Here, we present the main steps of the proof of Theorem  \ref{thm1.7.2}, postponing their proofs until Sections 5 and 6.
Our goal is to show that (\ref{eq3.14.2}) of Proposition \ref{prop3.14.1} is satisfied for each SRBM satisfying the
conditions of Theorem \ref{thm1.7.2}.  It then follows from the proposition 
that the SRBM is positive recurrent.  

We employ the notation $D_1,D_2,\ldots$ and
$\epsilon_1,\epsilon_2,\ldots$, as well as  the previous notation $C_1,C_2,\ldots$, 
to denote positive constants.  As earlier, $C_i$ denote terms
whose precise value is not of interest to us, with the same symbol sometimes being reused.  
The terms $D_i$ and $\epsilon_i$ will sometimes take general values in the
statements of the propositions, in which case specific values will be employed at the end of the section 
to demonstrate (\ref{eq3.14.2}).  We state the values of $D_i$ and $\epsilon_i$ we will apply, in most cases, 
when they are first introduced.

Proposition \ref{prop4.2.1} is the first result.  It states in essence that, after an appropriate time, either the norm
of the initial state of the process decreases by a large factor or the sixth coordinate
is bounded away from $0$.  In the first case,
(\ref{eq3.14.2}) will be demonstrated by using Proposition \ref{prop4.5.1}.  In the second case, this will be done
by using Propositions \ref{prop4.6.1}-\ref{prop4.8.1}.
In the statement of Proposition \ref{prop4.2.1}, one can choose $\ssb = 24\cdot16\cdot4+4$ and 
$\ssd = 24\cdot16\cdot40/\delta_1\delta_3$.  At the end of the section, we will set
$\epsa = \epsagg^2$;
the term $\epsagg \in (0,\delta_1\delta_2\delta_3/1200]$, with the exact value being specified then.
\begin{prop}
\label{prop4.2.1}
Suppose that $Z(0) = z$ with $z_1\le M$, $z_{k}^{2} \le M$, for $k=2,\ldots,5$, and $z_6 \le M$.  \text{(a)} For given 
$\epsa > 0$, there exist $\ssc$, $\ssb \ge 1$ and $\epsilon^{\prime} > 0$ such that, for all $M$,
\begin{align}
\label{eq4.2.2}
& P(Z_k(M) > \ssb M) \le \ssc e^{-\epsilon^{\prime}M} \qquad \text{for } k=1,6, \\
\label{eq4.2.3}
& P(Z_k(M) > \epsa M) \le \ssc e^{-\epsilon^{\prime}M} \qquad \text{for } k=2\ldots,5, \\
\label{eq4.2.4}
& E[Z_k(M)^2] \le \ssb M  \qquad \text{for } k=2,\ldots,5.
\end{align}
\text{(b)} For appropriate $\ssd > 0$ and each $\epsagg \in (0,\tfrac{1}{40}\delta_1\delta_3]$, 
there exist sets $F_1 \in \mathcal{F}(M)$, $F_2 \in \mathcal{F}(M)$ and 
$\epsilon^{\prime}>0$ such that, for large enough $M$,
\begin{align}
\label{eq4.2.5}
& P((F_1 \cup F_2)^c) \le  e^{-\epsilon^{\prime}M} \qquad \text{for } k=1,6, \\
\label{eq4.2.6}
& Z_6(M) \le \epsagg M \text{ on } F_1, \\
\label{eq4.2.7}
& E[Z_k(M)^2; F_1] \le \epsagg \ssd M  \qquad \text{for } k=2,\ldots,5, \\
\label{eq4.2.8}
& Z_6(M) \ge \epsagg M \text{ on } F_2.
\end{align}
\end{prop}

Depending on whether
$F_1$ or $F_2$ holds, we proceed in different ways.  Under $F_1$, we consider the evolution of the SRBM for an
additional time $\sse M$.  For this, we employ Proposition \ref{prop4.5.1}, which is given below.

We introduce the following terminology for Proposition \ref{prop4.5.1}.  Set
\begin{equation}
\label{eq4.4.1}
\epsb = 6 \epsagg / \delta_3 \quad \text{ and } \quad \sse = 250 \ssb / \delta_3 \delta_4,
\end{equation}
for given $\epsagg > 0$ and $\ssb \ge 1$.  Let $U_1$ be the first time $t$ on the interval
$[0,(\sse - \epsb)M]$ at which $Z_1(t) = \frac{12 \epsagg}{\delta_3}M$,
with $U_1 = (\sse - \epsb)M$ if this does not occur.  Set $U_2 = U_1 + \epsb M \le \sse M$.
The proposition states that $Z_1(U_2)$, $Z_6(U_2)$ and $Z_k(U_2)$, $k=2,\ldots,5$, are all small in an
appropriate sense.  The argument requires $Z_1(t) > 0$ for $t \le U_2$, which enables all other coordinates to drift
toward $0$.  
\begin{prop}
\label{prop4.5.1}
Suppose $Z(0) = z$ satisfies
\begin{align}
\label{eq4.5.2}
&  (12\epsagg /\delta_3)M \le z_1 \le \ssb M, \\   
\label{eq4.5.2'}
&z_k \le  \epsagg M, \qquad \qquad \qquad  \text{for } k = 2,\ldots,6,
%
%
%
\end{align}
for given $\ssb \ge 1$ and $\epsagg \in (0,1]$.  Then, for $U_2$ as given above and large enough $M$,
\begin{align}
\label{eq4.5.4}
& E[Z_k(U_2)] \le (70 \epsagg /\delta_3)M \qquad \text{for } k=1,6, \\
\label{eq4.5.5}
& E[Z_k(U_2)^2] \le (24 \cdot 97 \epsagg /\delta_3)M \qquad \text{for } k=2,\ldots,5.
\end{align}
\end{prop}

When $F_2$ occurs, we follow the sketch given near the end of Section 1.  In this case, we restart the SRBM at time
$M$ and apply Proposition \ref{prop4.6.1}.  In the proposition, we employ the stopping times 
$T_3(\cdot)$ and $T_4(\cdot)$.  We define
\begin{equation}
\label{eq4.5.6}
T_3(M) = \min \left\{t: \sum_{k=2}^{5} Y_k(t) = \tfrac{1}{6} (\delta_1 t + \epsg M) \right\}
\end{equation}
for given $\epsg \in (0,1]$. 
We then set $T_4(M) = T_3(M)$ off of a set $G_M$ that will be specified in the proof of the proposition,
with $T_4(M) \le T_3(M) \wedge T_6$ holding on $G_M$, where $T_6$ is the stopping time that was defined in (\ref{eq3.3.0b}).  
($T_3(M) < T_6$ will hold off of $G_M$.)  The set $G_M$ will be negligible in the sense of 
(\ref{eq4.6.9}) and (\ref{eq4.6.10}).

In addition to the bounds on $G_M$ in (\ref{eq4.6.9}) and (\ref{eq4.6.10}), Proposition \ref{prop4.6.1} 
gives upper and lower bounds on $T_3(M)$ and $Z_k(T_3(M))$, for $k=1,6$, and upper bounds on $Z_k(T_3(M))$, 
for $k=2,\ldots,5$.  
We will set the constant $\epsd$ in the proposition equal to $\frac{1}{10}$ 
at the end of the section.
\begin{prop}
\label{prop4.6.1}
Suppose $Z(0) = z$ satisfies
\begin{align}
\label{eq4.6.2}
& z_k \le \ssb M \qquad \text{for } k=1,6, \\
\label{eq4.6.3}
& z_k \le \epsg^2 M \qquad \text{for } k=2,\ldots,5, \\
\label{eq4.6.4}
& z_6 \ge \epsg M,
\end{align}
for given $M$, $\ssb$ and $\epsg \in  (0, \tfrac{1}{20}]$.  Then,
on $G_{M}^{c}$ and $T_3(M) < \infty$,
\begin{align}
\label{eq4.6.5}
& T_3(M) \ge \tfrac{1}{30}\epsg M, \\
\label{eq4.6.5'}
& T_3(M) \le T_6, \\
\label{eq4.6.6}
& Z_k(T_3(M)) \le \tfrac{31 \ssb}{\epsg}T_3(M) \qquad \text{for } k=1,6, \\
\label{eq4.6.8}
& Z_k(T_3(M)) \le 31\epsg T_3(M) \qquad \text{for } k=2,\ldots,5, \\ 
\label{eq4.6.7}
& Z_1(T_3(M)) \ge \tfrac{1}{12} \delta_1 \delta_2 T_3(M), \\ 
\label{eq4.6.7'}
& Z_6(T_3(M)) \ge \tfrac{1}{2}\delta_1 T_3(M).
\end{align}
For given $\epsd > 0$ and large enough $M$,
\begin{align}
\label{eq4.6.9}
& E[Z_k(T_4(M)); \, G_M] \le \epsd \qquad \text{for } k=1,6, \\
\label{eq4.6.10}
& E[Z_k(T_4(M))^2; \, G_M] \le \epsd \qquad \text{for } k=2,\ldots,5.
\end{align}
\end{prop}

We define stopping times $T_{3}^{\prime}(M)$ as follows.  For given $M>0$ and $z=(z_1,\ldots,z_6)$, set
\begin{equation}
\label{eq4.20.1}
T_{3}^{\prime}(M) = T_3(M) \wedge T_6 \wedge 5N_M(z)^{5/2\eta},
\end{equation}  
where
\begin{equation}
\label{eq4.20.2}
N_M(z) = \left(\max_{k=2,\ldots,5}z_{k}^{2}\right) \vee M
\end{equation}  
and $\eta = .000073$ as in Section 3.  Assuming random initial conditions 
that satisfy the analog of (\ref{eq4.2.4}) in 
Proposition \ref{prop4.2.1}, we give, in Proposition \ref{prop4.21.1}, bounds on $E[T_{3}^{\prime}(M)]$.
Moreover, the truncation event
\begin{equation}
\label{eq4.20.3}
A_M = \{\omega: T_{3}^{\prime}(M) = 5N_M(z)^{5/2\eta}\}
\end{equation}  
is small in the sense of (\ref{eq4.21.4}) and, under further initial conditions, is small as in
(\ref{eq4.21.4'}).  Propositions \ref{prop3.5.2} and \ref{prop3.7.1} are the key
ingredients in the proof. 
\begin{prop}
\label{prop4.21.1}
Suppose that $Z(0)$ satisfies 
\begin{equation}
\label{eq4.21.2}
E[Z_k(0)^2] \le \ssb M \qquad \text{for } k=2,\ldots,5,
\end{equation}  
and given $M \ge 4$ and $D_1$.  Then, for appropriate $\ssl$ not depending on $M$,
\begin{equation}
\label{eq4.21.3}
E[T_{3}^{\prime}(M)] \le \ssl M 
\end{equation}
and
\begin{equation}
\label{eq4.21.4}
E[Z_k(T_{3}^{\prime}(M))^2; \, A_M] \le \ssl /\sqrt{M}. 
\end{equation}
If, in addition, $Z_k(0) \le \ssb M$ for $k=1,6$, then
\begin{equation}
\label{eq4.21.4'}
E[Z_k(T_{3}^{\prime}(M)); \, A_M] \le \ssm /M
\end{equation}
for appropriate $\ssm$ not depending on $M$.
\end{prop}

On the set $G_{M}^{c} \cap A_{M}^{c}$, we continue to follow the evolution of $Z(\cdot)$ after the 
elapsed time $M + T_3(M)$.  (Note that, on $G_{M}^{c} \cap A_{M}^{c}$, $T_3(M) = T_{3}^{\prime}(M)$.)  
We wish to show that, provided $Z_k(\cdot)$, $k=1,6$, are initially ``large"
but $Z_k(\cdot)$, $k=2,\ldots,5$, are initially ``small", then all coordinates will typically be small at an
appropriate random time.  This is done in Proposition \ref{prop4.8.1}.  The bounds (\ref{eq4.8.6}) 
and (\ref{eq4.8.7})  will allow us to 
demonstrate (\ref{eq3.14.2}) under the event $F_2$ in 
Proposition \ref{prop4.2.1}.

In order to state Proposition \ref{prop4.8.1}, we define
\begin{align}
\label{eq4.7.1}
& T_5(M_1) = \min \{t: Z_1(t) = \epsk M_1\} \wedge \ssh M_1, \\
\label{eq4.7.2}
& T_{5}^{\prime}(M_1) = T_5(M_1) + \tfrac{1}{2}\epsk M_1,
\end{align}
for given $M_1 > 0$, $\ssh$ and $\epsk > 0$.  Note that 
\begin{equation}
\label{eq4.7.3}
T_{5}^{\prime}(M_1) \le (\ssh + \tfrac{1}{2}\epsk)M_1 
\end{equation}  
always holds.  We employ the constants $\epsk$, $\epsp$, $\epsq$, $\epsr$ and $\ssh$, $\ssg$ in the proposition.
A specific value of $\epsk\in (0,\tfrac{1}{72}\delta_1\delta_2]$ will be assigned at the end of the section; 
there, we will also employ $\epsp = 31 \epsg$, $\epsq = \tfrac{1}{12}\delta_1\delta_2$,
$\epsr = \tfrac{1}{20}\delta_3 \epsk$ and $\ssg = 31\ssb/\epsg$; $\ssh$ is specified
in the proposition. 
\begin{prop}
\label{prop4.8.1}
Let $T_5(\cdot)$ and $T_{5}^{\prime}(\cdot)$ be as in (\ref{eq4.7.1}) and (\ref{eq4.7.2}) for given $\epsk >0$.
Suppose $Z(0)=z$ satisfies
\begin{align}
\label{eq4.8.2}
& z_k \le \epsp M_1 \qquad \text{for } k=2,\ldots,5, \\
\label{eq4.8.3}
& \epsq M_1 \le z_k \le \ssg M_1 \qquad \text{for } k=1,6, 
\end{align}
for given $M_1 > 0$, $\epsp > 0$, $\epsq \ge 6\epsk \vee 3\epsp$ and $\ssg > 0$.  Then, for given
$\epsr > 0$ and $\ssh = 10\ssg \delta_4/\delta_2\delta_3$,
\begin{align}
\label{eq4.8.4}
& P(T_5(M_1) = \ssh M_1) \le \ssi e^{-\epsilon^{\prime}M_1}, \\
\label{eq4.8.5}
& P(Z_k(T_5(M_1)) \ge \epsr M_1) \le \ssi e^{-\epsilon^{\prime}M_1} \qquad \text{for } k=2,\ldots,6,
\end{align}
for appropriate $\ssi$ and $\epsilon^{\prime} > 0$ not depending on $M_1$.  Moreover,
\begin{align}
\label{eq4.8.6}
& E[Z_k(T_{5}^{\prime}(M_1))] \le 6\epsk M_1 + \ssj \qquad \text{for } k=1,6, \\
\label{eq4.8.7}
& E[Z_k(T_{5}^{\prime}(M_1))^2] \le 24\cdot8 \epsk M_1 + \ssj \qquad \text{for } k=2,\ldots,5,
\end{align}
for appropriate $\ssj$ not depending on $M_1$.
\end{prop}

\subsection*{Demonstration of Theorem \ref{thm1.7.2}}
It suffices to consider the SRBM $Z(\cdot)$ on the induced $Z$-path space.  
We will show that, for $z\in \mathbb{R}_{+}^{6}$ and an appropriate stopping time $\sigma(z)$, 
the assumption (\ref{eq3.14.2}) of Proposition \ref{prop3.14.1} is satisfied. The proposition will then
imply $Z(\cdot)$ is positive recurrent.  We abbreviate by setting $\sigma(z) = \sigma$ and dropping 
the subscript $z$ from $E_z[\cdot]$.  

We will express $\sigma$ in terms of a related stopping time $\sigma^{\prime}$, which we construct
piecemeal by using the sets appearing in the previous propositions. 
Assume that $\|z\| = M$.  Then $z_1 \le M$, $z_{k}^{2} \le M$, for $k=2,\ldots,5$, and $z_6 \le M$, and so the
assumptions of Proposition \ref{prop4.2.1} are satisfied.  It follows from the proposition that 
(\ref{eq4.2.2})--(\ref{eq4.2.4}) hold for given $M$ and (\ref{eq4.2.5})--(\ref{eq4.2.8}) hold for large enough $M$.  
Let $\hha$ denote the union of the set where $(F_1 \cup F_2)^c$ occurs and where either the event in (\ref{eq4.2.2}) or
the event in (\ref{eq4.2.3}) occurs.  On $\hha$, we set $\sigma^{\prime} = M$.  It follows from Lemma \ref{lem2.7.1},
(\ref{eq4.2.2}), (\ref{eq4.2.3}) and (\ref{eq4.2.5}) that, for large enough $M$,
\begin{equation}
\label{eq4.10.5}
E[\|Z(\sigma^{\prime})\|; \, \hha] \le 1.
\end{equation} 

Suppose next that the event $F_1 \cap \hha^c$ holds.  Then either (a) $Z_1(M) < (12\epsagg/\delta_3)M$ or
(b) $Z_1(M) \ge (12\epsagg/\delta_3)M$; denote the former of these events by $\hhb$ and the latter by $\hhc$.
Under $\hhb$, we set $\sigma^{\prime} = M$.  Then, on account of (\ref{eq4.2.6}) and (\ref{eq4.2.7}) of Proposition
\ref{prop4.2.1}, with $\ssd = 24\cdot16\cdot40/\delta_1\delta_3$,
\begin{equation}
\label{eq4.11.1}
\begin{split}
E[\|Z(\sigma^{\prime})\|; \, \hhb] & \le \left(\tfrac{12\epsagg}{\delta_3} + 4 \epsagg \ssd + \epsagg \right)M \\
& \le (97\cdot16\cdot40 \epsagg/\delta_1 \delta_3) M.
\end{split}
\end{equation}
When $\hhc$ occurs, we set $\sigma^{\prime} = M + U_2$, where $U_2$ is defined below (\ref{eq4.4.1}).
(Here and later on, stopping times such as $U_2$ refer to the restarted process.)  The
process restarted at time $M$ satisfies conditions (\ref{eq4.5.2}) and (\ref{eq4.5.2'}) of Proposition
\ref{prop4.5.1}.  It follows from (\ref{eq4.5.4}) and (\ref{eq4.5.5}) of the proposition that
\begin{equation}
\label{eq4.11.2}
E[\|Z(\sigma^{\prime})\|; \, \hhc] \le (140 + 96\cdot97)\frac{\epsagg}{\delta_3}M \le 97^2 \frac{\epsagg}{\delta_3}M.
\end{equation}

The bounds (\ref{eq4.10.5})--(\ref{eq4.11.2}) consider the behavior of $Z(\sigma)$ off of $F_2 \cap \hha^c$.
We now consider the behavior on $F_2 \cap \hha^c$, for which there are two cases.  Denote by $\hhd$ the subset
of $F_2 \cap \hha^c$ corresponding to the union of the events $G_M$ and $A_M$ for the restarted process, 
which appear in the proof of Proposition \ref{prop4.6.1} and in (\ref{eq4.20.3}).  Let 
$$\sigma^{\prime} \stackrel{\text{def}}{=} M + (T_4(M) \wedge 5N_{M}(Z(M))^{5/\eta}) \le M + T_{3}^{\prime}(M),$$ 
that is, $\sigma^{\prime}$ is the earlier of the times at which either the event $G_M$ or $A_M$
occurs.   The restarted process satisfies both (\ref{eq4.6.2})--(\ref{eq4.6.4}) of Proposition \ref{prop4.6.1} and
(\ref{eq4.21.2}) of Proposition \ref{prop4.21.1}.  It therefore follows from (\ref{eq4.6.9})--(\ref{eq4.6.10}),
with $\epsd = \tfrac{1}{10}$, and (\ref{eq4.21.4})--(\ref{eq4.21.4'}) that
\begin{equation}
\label{eq4.11.3}
E[\|Z(\sigma^{\prime})\|; \, \hhd] \le 1
\end{equation}   
for large enough $M$.

We also consider the behavior of $Z(\sigma^{\prime})$ on $\hhe \stackrel{\text{def}}{=} F_2 \cap \hha^{c} \cap \hhd^{c}$.
On account of (\ref{eq4.6.6})--(\ref{eq4.6.7'}) of Proposition \ref{prop4.6.1}, the conditions 
(\ref{eq4.8.2})--(\ref{eq4.8.3}) of Proposition \ref{prop4.8.1} are satisfied for the process restarted at time
$M + T_3(M) = M + T_{3}^{\prime}(M)$, for $M_1 = T_3(M)$ and $\ssg$, $\epsp$ and $\epsq$ 
as specified before Proposition \ref{prop4.8.1}. Also, $\epsq \ge 6\epsk \vee 3\epsp$ holds for
$\epsagg \le \delta_1\delta_2\delta_3/1200$ and $\epsk$ as specified before the proposition.  
Inequalities (\ref{eq4.8.6}) and (\ref{eq4.8.7}) therefore hold for 
$T_{5}^{\prime}(M_1)$ chosen as in (\ref{eq4.7.2}).  Setting $\sigma^{\prime} = M + T_{3}(M) + T_{5}^{\prime}(T_3(M))$, it
follows from these inequalities that
\begin{equation}
\label{eq4.12.2}
E[\|Z(\sigma^{\prime})\|; \, \hhe] \le 97\cdot8\epsk E[T_3(M); \, \hhe] + \ssj 
\le 97\cdot8 \epsk E[T_{3}^{\prime}(M)] + \ssj
\end{equation} 
for appropriate $\ssj$.  On account of (\ref{eq4.2.4}) of Proposition \ref{prop4.2.1}, 
one can apply Proposition \ref{prop4.21.1} to
$Z(\cdot)$ restarted at time $M$, which gives the upper bound in (\ref{eq4.21.3}) on $E[T_{3}^{\prime}(M)]$.
Applying this to (\ref{eq4.12.2}), one obtains 
\begin{equation}
\label{eq4.13.1}
E[\|Z(\sigma^{\prime})\|; \, \hhe] \le 98\cdot8 \epsk \ssl M
\end{equation}
for large enough $M$ and appropriate $\ssl$.

Adding the bounds in (\ref{eq4.10.5})--(\ref{eq4.13.1}) for $E[Z(\sigma^{\prime}); H_i]$, 
$i = 1,\ldots,5$, one obtains
\begin{equation*}
E[\|Z(\sigma^{\prime})\|] \le \eee(\epsagg + \epsk)M
\end{equation*}
for large enough $M$, with $\eee$ depending on $\delta_1$ and $\delta_3$. So far, we have not specified the values of
$\epsagg$ and $\epsk$; we now set 
$$\epsagg = \epsk = (1/4\eee)\wedge (\delta_1\delta_2\delta_3/1200).$$  
It follows that
\begin{equation}
\label{eq4.13.2}
E[\|Z(\sigma^{\prime})\|] \le \tfrac{1}{2}M
\end{equation}
for $\|z\| = M$ and $M\ge M_0$, for appropriate $M_0 \ge 1$.

We define $\sigma$ in terms of $\sigma^{\prime}$, by setting $\sigma = \sigma^{\prime}$ when $\|z\|=M$ for 
$M \ge M_0$, and $\sigma = M \vee 1$ for $M\le M_0$. 
When $\|z\| = M$ and $M \ge M_0$, this implies
\begin{equation}
\label{eq4.13.3}
E[\|Z(\sigma)\|] \le \tfrac{1}{2}M.
\end{equation}
On the other hand, by applying (\ref{eq2.7.3}) of Lemma \ref{lem2.7.1} to (\ref{eq4.2.2}),
it follows for all $M$ that
$$ E[Z_k(M\vee1)] \le \aaa (M\vee1) \qquad \text{for } k=1,6 $$
and appropriate $\aaa \ge \ssb \vee 1$.  Together with ($\ref{eq4.2.4}$), this implies
\begin{equation}
\label{eq4.14.1}
E[\|Z(M \vee 1)\|] \le 6 \aaa (M \vee 1)
\end{equation}
for all $M$.  Setting $\kappa = 12 \aaa(M_0 \vee 1)$, it follows from (\ref{eq4.13.3}) and (\ref{eq4.14.1}) that
\begin{equation}
\label{eq4.14.2}
E[\|Z(\sigma)\|] \le (\|z\| \vee \kappa) - \tfrac{1}{2}(M \vee 1)
\end{equation}
for $\|z\|=M$ and all $M$.

We also wish to show that, for $\|z\|=M$,
\begin{equation}
\label{eq4.14.3}
E[\sigma] \le \ccc (M \vee 1)
\end{equation}
for some $\ccc$.  This is a quick consequence of the definition of $\sigma$ on $\hha,\ldots,\hhe$ for $\|z\| \ge M_0$.  
On $\hha \cup \hhb$, $\sigma = M$; on $\hhc$, $\sigma \le \sse M$; on $\hhd$, $\sigma \le M + T_{3}^{\prime}(M)$;
and on $\hhe$, $\sigma = M + T_{3}^{\prime}(M) + T_{5}^{\prime}(T_{3}^{\prime}(M))$.  It therefore follows from
(\ref{eq4.21.3}) of Proposition \ref{prop4.21.1} and (\ref{eq4.7.3}) that
\begin{equation}
\label{eq4.15.1}
\begin{split}
E[\sigma] & \le M + \sse M + E[T_{3}^{\prime}(M)] + E[T_{5}^{\prime}(T_{3}^{\prime}(M))] \\
& \le (1 + \sse + \ssl + \ssl (\ssh + \tfrac{1}{2}\epsk))M \le \ccc M
\end{split}
\end{equation}  
for $\|z\| \ge M_0$ and appropriate $\ssl$ and $\ccc$.  Together with $\sigma = M \vee 1$ for $\|z\| < M_0$, this 
implies (\ref{eq4.14.3}).  

Combining (\ref{eq4.14.2}) and (\ref{eq4.14.3}), one obtains
\begin{equation}
\label{eq4.15.2}
E[\|Z(\sigma)\|] \le (\|z\| \vee \kappa) - (1/2\ccc) E[\sigma].
\end{equation}
This implies (\ref{eq3.14.2}) of Proposition \ref{prop3.14.1}, with $\epsilon = 1/2\ccc$.
Since $Z(\cdot)$ is Feller and $\sigma$ is defined in terms of hitting times of closed sets, one can
check that $E_z[\sigma(z)] = E[\sigma]$ is measurable in $z$.
By applying the proposition, (\ref{eq3.14.3}) follows and hence $Z(\cdot)$ is positive recurrent.  This 
demonstrates Theorem \ref{thm1.7.2}.  

\section{Demonstration of Propositions \ref{prop4.2.1} and \ref{prop4.5.1}}
Proposition \ref{prop4.2.1} constitutes the first step of the proof of Theorem \ref{thm1.7.2}.
It provides elementary upper bounds (\ref{eq4.2.2})--(\ref{eq4.2.4}) on $Z_k(M)$, $k=1,\ldots,6$, and on $E[Z_k(M)^2]$, 
$k=2,\ldots,5$, that are valid over all $M$.  It states that, off of the exceptional set in  (\ref{eq4.2.5}), either $Z_k(M)$
will be small for all $k=2,\ldots,6$ or $Z_6(M)$ will be large, 
in the sense of (\ref{eq4.2.6})--(\ref{eq4.2.8}).  This dichotomy depends on the
rate of growth of $Y_1(\cdot)$ as given by the set $\ffc$ in (\ref{eq5.1.1}) of the proof, although
the actual correspondence is a bit more complicated.  The proof of Proposition \ref{prop4.2.1} relies on the application of
lemmas from Section 2 to the equation (\ref{eq1.2.1}) of the SRBM.
\begin{proof} [Proof of Proposition \ref{prop4.2.1}]
Both inequalities in (\ref{eq4.2.2}) follow directly from (\ref{eq2.4.4}) of Lemma \ref{lem2.4.2}, with 
$\ssb \ge 9$.  Inequality (\ref{eq4.2.3}) follows from (\ref{eq2.5.4}) of Lemma \ref{lem2.5.2}, with a new
choice of $\aaa$.  For (\ref{eq4.2.4}), one can restrict the expectation to the set $\{Z_k(M) > 2\sqrt{M}\}$ and its
complement.  One then applies (\ref{eq2.6.1}) to the first part and a trivial bound to the second part to obtain
(\ref{eq4.2.4}), with $\ssb \ge 24 \cdot 16 \cdot 4 + 4$.

For the inequalities (\ref{eq4.2.5})--(\ref{eq4.2.8}), we first set 
\begin{equation}
\label{eq5.1.1}
\ffc = \{\omega: Y_1(M) - Y_1(\tau_k) > \epsc M \text{ for some } k=2,\ldots,5\}.
\end{equation}
Here, $\epsc \stackrel{\text{def}}{=} 2 \epsagg / \delta_1$ and $\tau_k$ is the last time before $M$ at which
$Z_k(t) = 0$ for any $t$; if the set is empty, let $\tau_k = 0$.  On $\ffc$, we denote by $K$ one of the indices
$k$ satisfying (\ref{eq5.1.1}).  

We consider the behavior on $\ffc$ and $\ffc^{c}$ separately, first considering the behavior on $\ffc$.  One has, by
applying (\ref{eq1.2.1}) to the $K^{\text{th}}$ and $6^{\text{th}}$ coordinates, 
\begin{equation}
\label{eq5.2.1}
\begin{split}
Z_6(M) - Z_K(M)  = (& Z_6(\tau_K) - Z_K(\tau_K)) \\ 
& + (B_6(M) - B_6(\tau_K)) - (B_K(M) - B_K(\tau_K)) \\
& + \delta_1 (Y_1(M) - Y_1(\tau_K)) + \delta_3(Y_6(M) - Y_6(\tau_K)).
\end{split}
\end{equation}
On $\ffc$, it follows from (\ref{eq2.1.3}) of Lemma \ref{lem2.1.1} that, except on a 
set $\ffd \in \mathcal{F}(M)$ of exponentially small probability in $M$,
\begin{equation}
\label{eq5.2.2}
Z_6(M) \ge \delta_1 \epsc M - \epsilon M \ge \tfrac{1}{2} \delta_1 \epsc M = \epsagg M
\end{equation}
for $\epsilon = \tfrac{1}{2}\delta_1 \epsc$ and large enough $M$.  This gives the inequality in (\ref{eq4.2.8}) on
the set $\ffc \cap \ffd^{c}$.

We now consider the behavior of $Z(\cdot)$ on $\ffc^{c}$.  Set $t_1 = (1 - 20\epsc/\delta_3)M$; since
$\epsagg \le \tfrac{1}{40}\delta_1\delta_3$, $t_1 \ge 0$ holds.   It follows from
(\ref{eq2.5.4}) of Lemma \ref{lem2.5.2} that, except on a set $\ffe \in \mathcal{F}(M)$ 
of exponentially small probability in $M$,
\begin{equation}
\label{eq5.2.3}
Z_k(t_1) \le Z_k(0) + \epsilon M \le \epsc M
\end{equation}
for $k=2,\ldots,5$, $\epsilon = \epsc/2$ and large enough $M$.  Restarting $Z(\cdot)$ at time $t_1$, it follows from
(\ref{eq2.10.3}) of Lemma \ref{lem2.10.2} and (\ref{eq5.2.3}) that, except on a set $\ffg \in \mathcal{F}(M)$ of
exponentially small probability,
\begin{equation}
\label{eq5.3.1}
Y_k(M) - Y_k(t_1) \ge 4\frac{\epsc}{\delta_3}M - \frac{\epsc}{\delta_3}M - \frac{2}{\delta_3}(Y_1(M) - Y_1(t_1)).
\end{equation}
On $\ffc^c$, when $\tau_k < t_1$, the last term on the right side of (\ref{eq5.3.1})  is at 
most $2\epsc M/\delta_3$, which implies
\begin{equation}
\label{eq5.3.1'}
 Y_k(M) - Y_k(t_1) > 0, 
\end{equation}
and hence $Z_k(\tau_{k}^{\prime}) = 0$ for some $\tau_{k}^{\prime} \in [t_1,M]$.  This contradicts the definition of
$\tau_k$, and so $\tau_k \ge t_1$.

Let $\tau_{k}^{\prime}$ be the smallest such time.  Since $\tau_{k}^{\prime}$ is a stopping time, we may restart
$Z(\cdot)$ at $\tau_{k}^{\prime}$.  Applying (\ref{eq2.6.1}) of Lemma \ref{lem2.5.2}, it follows that 
\begin{equation}
\label{eq5.3.2}
E[Z_k(M)^2; \ffc^{c} \cap \ffe^{c} \cap \ffg^{c}] \le 24 \cdot 16 \cdot 20 \frac{\epsc}{\delta_3}M = \epsagg \ssd M
\end{equation}
for $k = 2,\ldots,5$ and $\ssd = 24 \cdot 16 \cdot 40 / \delta_1 \delta_3$.

We now conclude the demonstration of (\ref{eq4.2.5})--(\ref{eq4.2.8}).  Denoting the set on which the inequality
in (\ref{eq4.2.8}) holds by $F_2$, one has by (\ref{eq5.2.2}) that $F_2 \supseteq \ffc \cap \ffd^{c}$.  Setting
$F_1 = F_{2}^{c} \cap \ffc^{c} \cap \ffe^{c} \cap \ffg^{c}$, then (\ref{eq4.2.6}) is automatically satisfied and (\ref{eq4.2.7})
holds because of (\ref{eq5.3.2}).  Since $(F_1 \cup F_2)^c \subseteq  \ffd \cup \ffe \cup \ffg$, 
(\ref{eq4.2.5}) follows, for appropriate $\epsilon^{\prime} > 0$, from the upper bounds on the probabilities of
$\ffd$, $\ffe$ and $\ffg$.  It follows from the definition of $F_2$ that $F_2 \in \mathcal{F}(M)$; since 
$F_i \in \mathcal{F}(M), i = 2,\ldots,6$, one also has $F_1 \in \mathcal{F}(M)$. 
\end{proof}

Proposition \ref{prop4.5.1} states that, if $z_k$, $k=2,\ldots,6$, are all small and $z_1$ is bounded below, but is not too
large, then $Z_k(U_2)$, $k=1,\ldots,6$, are all small in the sense of (\ref{eq4.5.4}) and (\ref{eq4.5.5}).  The proof 
considers the behavior of $Z(t)$ over $[U_1,U_2]$.  The stopping time $U_1$ was defined so that $Z_1(U_1)$ is relatively
small, but large enough so that, over $[0,U_2]$ with $U_2=U_1 + \epsb M$, $Z_1(t) > 0$ holds.  The interval $[U_1,U_2]$
is both large enough to obtain the desired behavior of $Z_k(U_2)$, $k=2,\ldots,5$, in (\ref{eq4.5.5}) and short enough so
(\ref{eq4.5.4}) holds for $Z_k(U_2)$, $k=1,6$.  As with Proposition \ref{prop4.2.1}, the proof applies the lemmas of
Section 2 to (\ref{eq1.2.1}).

\begin{proof}[Proof of Proposition \ref{prop4.5.1}]
We first show (\ref{eq4.5.4}) for $k=1$. 
It follows from Lemma \ref{lem2.12.1}, (\ref{eq4.5.2}) and (\ref{eq4.5.2'}) that, on the set where $Z_1(t) > 0$
for $t\in [0, \tfrac{1}{2}\sse M]$,
\begin{equation}
\label{eq5.5.1}
Z_1(\tfrac{1}{2} \sse M) \le \ssb M + \tfrac{3 \epsagg}{\delta_3}M - \tfrac{1}{60}\delta_4 \sse M
\end{equation}
except for a set $\ffh$ of exponentially small probability in $M$.  Since the right side of (\ref{eq5.5.1}) is
negative for $\sse$ satisfying  (\ref{eq4.4.1}) and $Z_1(0)\ge \tfrac{12\epsagg}{\delta_3}M$, 
$Z_1(t) = \tfrac{12\epsagg}{\delta_3}M$ must occur at some $t\le \tfrac{1}{2}\sse M$;   
%
%
%
hence $Z_1(U_1) = \tfrac{12\epsagg}{\delta_3}M$ on $\ffh^c$.  
By (\ref{eq2.4.4}) of Lemma \ref{lem2.4.2} and (\ref{eq4.4.1}),
this in turn implies that $Z_1(U_2) \le \tfrac{60\epsagg}{\delta_3}M$ off of an additional set of exponentially
small probability.  Together with (\ref{eq2.7.3}) of Lemma \ref{lem2.7.1}, this implies (\ref{eq4.5.4}) for
$k=1$ and large $M$.

Restarting $Z(\cdot)$ at $U_1$, it follows from (\ref{eq1.2.1}) and (\ref{eq4.4.1}) that, except on a set $\ffi$ of
exponentially small probability in $M$,
\begin{equation*}
Z_1(U_1+s)) \ge \tfrac{12\epsagg}{\delta_3}M + B_1(s) -s > 0
\end{equation*}
for $s\le \epsb M$.  Consequently, on $\ffi^c$, 
\begin{equation}
\label{eq5.6.2}
Z_1(t) > 0 \qquad \text{for } t\le U_2.
\end{equation}
Since $Z_6(0) \le \epsagg M$, one can therefore employ (\ref{eq2.5.4}) of Lemma \ref{lem2.5.2}, with small enough
$\epsilon >0$, together with (\ref{eq2.7.3}) of Lemma \ref{lem2.7.1}, to obtain (\ref{eq4.5.4}) for $k=6$.

We still need to show (\ref{eq4.5.5}).  For this, one can employ the conditions (\ref{eq4.4.1}), (\ref{eq4.5.2'}) and
(\ref{eq5.6.2}) and argue similarly to (\ref{eq5.2.3}) through (\ref{eq5.3.1'}), in the proof of Proposition
\ref{prop4.2.1}, to conclude that, for $k=2,\ldots,5$,
$$ Z_k(\tau_{k}^{\prime}) = 0 \quad \text{ for some } \tau_{k}^{\prime} \in [U_1,U_2],$$
off of a set $\ffj$ of exponentially small probability in $M$.  Letting $\tau_{k}^{\prime}$ denote the first such
time, we restart $Z(\cdot)$ at $\tau_{k}^{\prime}$.  Applying (\ref{eq2.6.1}) and (\ref{eq2.7.2}), it follows that
\begin{equation*}
\begin{split}
E[Z_k(U_2)^2] & \le E[Z_{k}(U_2)^2; \,  \ffj^c] + E[Z_k(U_2)^2; \,  \ffj] \\
& \le (24\cdot 16 +1)\epsb M \le (24\cdot 97 \epsagg /\delta_3)M
\end{split}
\end{equation*}
for large enough $M$.  This implies (\ref{eq4.5.5}).
\end{proof}

\section{Demonstration of Propositions \ref{prop4.6.1}, \ref{prop4.21.1} and \ref{prop4.8.1}}
The proofs of Propositions \ref{prop4.6.1}, \ref{prop4.21.1} and \ref{prop4.8.1} rely on the application
of the lemmas in Section 2 to the equation (\ref{eq1.2.1}) of the SRBM $Z(\cdot)$.  Proposition \ref{prop4.21.1} also
relies on Propositions \ref{prop3.5.2} and \ref{prop3.7.1}.  The reasoning behind the proofs follows in spirit the
sketch given near the end of Section 1 and in Section 4.

We first demonstrate Proposition \ref{prop4.6.1}.  The proposition states that, off of the exceptional set $G_M$ defined
in the proof, the inequalities (\ref{eq4.6.5})--(\ref{eq4.6.7'}) all hold.  In particular, $Z_k(T_3(M))$, $k=2,\ldots,5$,
will be small and $Z_k(T_3(M))$, $k=1,6$, will be bounded below, but not too large.  These inequalities, except for
(\ref{eq4.6.7}), will follow from their analogs (\ref{eq6.1.1})--(\ref{eq6.1.3}) that hold over 
$[\tfrac{1}{30}\epsg M, T_3(M)]$ and $[0,T_3(M)]$.  The exceptional set $G_M$ will be shown to be 
small in the sense of (\ref{eq4.6.9}) and (\ref{eq4.6.10}).  

The lower bounds on $Z_k(T_3(M))$, $k=1,6$, constitute the more delicate part of the argument and depend on the 
condition $z_6\ge \epsg M$ in (\ref{eq4.6.4}).  Arguing as in (\ref{eq6.4.1})--(\ref{eq6.4.3'}), we will show that
the growth of $Y_1(\cdot)$ causes $Z_6(\cdot)$ to increase linearly.  
On the other hand, as shown below (\ref{eq6.1.6}), the
growth of $Y_k(\cdot)$, $k=2,\ldots,5$, together with $Y_6(T_3(M)) = 0$, causes $Z_1(\cdot)$ to eventually increase
linearly.  The stopping time $T_3(M)$ has been chosen so that both features are present.   
\begin{proof} [Proof of Proposition \ref{prop4.6.1}]
We first specify the set $G_M$ used in the definition of $T_4(M)$.  We abbreviate by
setting $M^{\prime} = \tfrac{1}{30}\epsg M$.  Writing 
$G_M = \cup_{i=1}^{5}G_i$,  the sets $G_i$ are defined as follows:  
\begin{align}
\label{eq6.1.1}
& \gga = \left\{\omega: \sum_{k=2}^{5}Y_k(M^{\prime}) \ge 5M^{\prime}\right\}, \\
\label{eq6.1.2}
& \ggb = \{\omega: Z_k(s) \ge \tfrac{31}{\epsg}\ssb s \text{ for some } s\in [M^{\prime}, T_3(M)],\, k=1,6 \}, \\
\label{eq6.1.4}
& \ggc = \{\omega: Z_k(s) \ge 31 \epsg s \text{ for some } s\in [M^{\prime}, T_3(M)],\, k=2,\ldots,5 \}, \\
\label{eq6.1.3}
& \ggd = \{\omega: Z_6(s) \le \tfrac{1}{2}\delta_1 s \text{ for some } s\in [0, T_3(M)] \}, \\
\label{eq6.1.5}
& \gge = \{\omega: B_2(s) - B_1(s) \ge \tfrac{1}{12}\delta_1 \delta_2 s \text{ for some } s\in [M^{\prime}, T_3(M)] \}.
\end{align}

Inequality (\ref{eq4.6.5}) follows from (\ref{eq6.1.1}) and the definition of $T_3(M)$.  Inequalities (\ref{eq4.6.6}) and 
(\ref{eq4.6.8}) follow by setting $s=T_3(M)$ in (\ref{eq6.1.2}) and (\ref{eq6.1.4}); both (\ref{eq4.6.5'}) 
and (\ref{eq4.6.7'}) follow from (\ref{eq6.1.3}).
The demonstration of (\ref{eq4.6.7}) requires a little work.  First note that,
on $\ggd^{c}$, (\ref{eq1.2.1}), (\ref{eq2.3.3}) of Lemma \ref{lem2.3.2} and (\ref{eq4.6.3}) imply that
\begin{equation}
\label{eq6.1.6}
\begin{split}
Z_1(T_3(M)) & \ge B_1(T_3(M)) - T_3(M) + \sum_{k=1}^{6} Y_k(T_3(M)) + \delta_2 \sum_{k=2}^{5}Y_k(T_3(M)) \\
& \ge \tfrac{1}{6}\delta_1 \delta_2 T_3(M) + (\tfrac{1}{6}\delta_2 \epsg - \epsg^2)M + B_1(T_3(M)) -B_2(T_3(M)) \\
& \ge \tfrac{1}{6}\delta_1 \delta_2 T_3(M) + B_1(T_3(M)) -B_2(T_3(M)),
\end{split}
\end{equation}
which, on $\gge^{c}$, is at least $\tfrac{1}{12}\delta_1 \delta_2 T_3(M)$.  So, 
\begin{equation*}
Z_1(T_3(M)) \ge \tfrac{1}{12}\delta_1 \delta_2 T_3(M) \qquad \text{on } \ggd^c \cap \gge^c,
\end{equation*}
which demonstrates (\ref{eq4.6.7}).

We need to show (\ref{eq4.6.9}) and (\ref{eq4.6.10}).  For this, we define $V_i$, $i=2,3,4,5$, to be the first time
at which the event in $G_i$ occurs, with 
$$ V_1 = \inf\left\{s: \sum_{k=2}^{5}Y_k(s) \ge 5M^{\prime} \right\} $$
%
if $\gga$ occurs; off of these sets, define $V_i = T_3(M)$ for
$i=1,\ldots,5$.  We complete our definition of $T_4(M)$ in (\ref{eq4.5.6}) by setting
\begin{equation}
\label{eq6.2.1}
T_4(M) = V_1 \wedge V_2 \wedge V_3 \wedge V_4 \wedge V_5.
\end{equation}
Note that $V_4 \le T_6$.
It follows from this and (\ref{eq6.2.1}) that $T_4(M) \le T_3(M) \wedge T_6$; moreover, $T_4(M)$ is a stopping time.

We note that, by (\ref{eq2.3.5}) of Lemma \ref{lem2.3.2},
\begin{equation}
\label{eq6.2.2}
P(\gga) \le \ssc e^{-\epsilon^{\prime}M}
\end{equation}
for appropriate $\ssc$ and $\epsilon^{\prime} > 0$.   Using (\ref{eq2.7.2}) and (\ref{eq2.7.3}) of Lemma \ref{lem2.7.1},
it therefore follows that, for given $\epsf > 0$ and large enough $M$,
\begin{align}
\label{eq6.2.3}
& E[Z_k(T_4(M)); \, \gga] \le \epsf \qquad \text{for } k=1,6, \\
\label{eq6.2.4}
& E[Z_k(T_4(M))^2; \, \gga] \le \epsf \qquad \text{for } k=2,\ldots,5.
\end{align}

We require more detailed estimates for $\ggb,\ldots,\gge$.
For each $i=2,\ldots,5$ and $j=1,2,\ldots$, we denote by $G_i(j)$ the event for which $G_i$ first occurs on $[j,j+1]$.  We
first consider the behavior on $\ggc$.  We recall that, by (\ref{eq2.5.4}) of Lemma \ref{lem2.5.2}, for 
$k=2,\ldots,5$ and given $\epsilon > 0$, 
\begin{equation}
\label{eq6.3.1}
P(Z_k(s) - Z_k(0) \ge \epsilon j \text{ for some } s\le j) \le \aaa e^{-\epsilon^{\prime}j}
\end{equation}
for each $j=1,2,\ldots$, and appropriate $\aaa$ and $\epsilon^{\prime} > 0$.  On account of (\ref{eq4.6.3}),
it follows for small enough $\epsilon$ that
\begin{equation}
\label{eq6.3.2}
P(Z_k(s) \ge 31\epsg s \text{ for some } s\in [j-1,j]) \le \aaa e^{-\epsilon^{\prime}j}
\end{equation}
for $j\ge M^{\prime}$.  It therefore follows from (\ref{eq2.7.2}) and (\ref{eq2.7.3}) that
\begin{align}
\label{eq6.3.3}
& E[Z_k(V_3); \, G_3(j)] \le e^{-\frac{1}{2}\epsilon^{\prime}j} \qquad \text{for } k=1,6, \\
\label{eq6.3.4}
& E[Z_k(V_3)^2; \, G_3(j)] \le e^{-\frac{1}{2}\epsilon^{\prime}j} \qquad \text{for } k=2,\ldots,5, 
\end{align}
for $j \ge M^{\prime}$ and large enough $M$.  Summing over $j$ gives
\begin{align}
\label{eq6.3.5}
& E[Z_k(V_i); \, G_i] \le \epsf \qquad \text{for } k=1,6, \\
\label{eq6.3.6}
& E[Z_k(V_i)^2; \, G_i] \le \epsf \qquad \text{for } k=2,\ldots,5,
\end{align}
for $i=3$, given $\epsf > 0$ and large enough $M$.

The inequalities (\ref{eq6.3.5}) and (\ref{eq6.3.6}) hold for $i=2$ and $i=5$ for the same reasons, except that one
applies (\ref{eq2.4.4}) of Lemma \ref{lem2.4.2} in place of (\ref{eq2.5.4}) and (\ref{eq4.6.2}) in place of
(\ref{eq4.6.3}) for $i=2$, and one applies (\ref{eq2.1.3}) of
Lemma \ref{lem2.1.1} for $i=5$.  The inequalities (\ref{eq6.3.5}) and (\ref{eq6.3.6}) also hold for $i=4$, although this
requires more work; we now do this.  

We note that, by (\ref{eq2.3.3}) of Lemma \ref{lem2.3.2}, (\ref{eq4.5.6}) and (\ref{eq4.6.3}),
\begin{equation}
\label{eq6.4.1}
Y_1(s) \ge s -  (\epsg^2 M + B_2(s)) - \tfrac{1}{6}(\delta_1 s + \epsg M)
\end{equation}
for $s\le T_3(M) \wedge T_6$.  Together with (\ref{eq1.2.1}), (\ref{eq4.6.4})
and $\epsg \le \tfrac{1}{20}$, this implies
\begin{equation}
\label{eq6.4.2}
\begin{split}
Z_6(s) & \ge \epsg M + B_6(s) - s + (1+\delta_1)Y_1(s) \\
& \ge \tfrac{3}{4}(\delta_1 s +\epsg M) + (B_6(s) - (1+\delta_1)B_2(s)).
\end{split}
\end{equation}
It follows from (\ref{eq2.1.2'}) of Lemma \ref{lem2.1.1} that
\begin{equation}
\label{eq6.4.3}
P(Z_6(s) \le \tfrac{1}{2}\delta_1 s \text{ for some } s\in [j-1,j]) \le \aaa e^{-\epsilon^{\prime}j}
\end{equation}
for $j\ge M^{\prime}$, and appropriate $\aaa$ and $\epsilon^{\prime} > 0$.  Also, by (\ref{eq2.1.3}),
\begin{equation}
\label{eq6.4.3'}
P(Z_6(s) \le \tfrac{1}{2}\delta_1 s \text{ for some } s\le M^{\prime}) 
\le \aaa e^{-\epsilon^{\prime}M}
\end{equation}
for appropriate $\aaa$ and $\epsilon^{\prime} > 0$.
Proceeding similarly to (\ref{eq6.3.2})--(\ref{eq6.3.6}), the inequalities
(\ref{eq6.3.5}) and (\ref{eq6.3.6}) with $i=4$ also hold.  

On account of (\ref{eq6.2.3})--(\ref{eq6.2.4}) and (\ref{eq6.3.5})--(\ref{eq6.3.6}), for $i=2,\ldots,5$, it follows
for large enough $M$ that
\begin{align}
\label{eq6.5.1}
& E[Z_k(T_4(M)); \, G_M] \le 5\epsf \qquad \text{for } k=1,6, \\
\label{eq6.5.2}
& E[Z_k(T_4(M))^2; \, G_M] \le 5\epsf \qquad \text{for } k=2,\ldots,5,
\end{align}
where $\epsf$ is as in (\ref{eq6.2.3})--(\ref{eq6.2.4}).  This implies (\ref{eq4.6.9}) and (\ref{eq4.6.10}) for
$\epsd = 5\epsf$, and completes the proof of the proposition.    

\end{proof}

The demonstration of Proposition \ref{prop4.21.1} is based on the comparison between $T_{3}^{\prime}(M,z)$ and
$T_2\left(\sqrt{N_M(z)}\right)$ in (\ref{eq6.20.1}).  This enables one to employ the upper bounds on
$E[T_2(x)]$ and $P(A)$ from Propositions \ref{prop3.5.2} and \ref{prop3.7.1}.
\begin{proof}[Proof of Proposition \ref{prop4.21.1}]
Let $T_{3}^{\prime}(M,z)$ and $A_M(z)$ denote the analogs of $T_{3}^{\prime}(M)$ and $A_M$, 
with $Z(0)=z$ being specified.
Comparing $T_{3}^{\prime}(M,z)$  with $T_2(x)$ in (\ref{eq3.5.1}), for $x= \sqrt{N_M(z)}$, it is easy to
see that
\begin{equation}
\label{eq6.20.1}
T_{3}^{\prime}(M,z) \le T_2\left(\sqrt{N_M(z)}\right).
\end{equation}
It therefore follows from Proposition \ref{prop3.5.2} that, for appropriate $\aaa$,
\begin{equation}
\label{eq6.20.2}
E[T_{3}^{\prime}(M,z)] \le E\left[T_2\left(\sqrt{N_M(z)}\right)\right] \le \aaa N_M(z). 
\end{equation}
Integrating (\ref{eq6.20.2}) over $z$ and applying (\ref{eq4.21.2}), one obtains
\begin{equation}
\label{eq6.20.3}
\begin{split}
E[T_{3}^{\prime}(M)] & = E[E[T_{3}^{\prime}(M) | \, Z(0) = z]] \le \aaa E[N_M(Z(0))] \\
& \le \aaa \left(E\left[\max_{k=2,\ldots,5}Z_k(0)^2\right] + M\right) \le \aaa (5\ssb +1)M.
\end{split}
\end{equation}
This implies (\ref{eq4.21.3}) with $\ssl = \aaa (5 \ssb +1)$.

Since the truncated values $T_6 \wedge 5N_M(z)^{5/\eta}$ and $T_6 \wedge 5x^{5/\eta}$ in
(\ref{eq4.20.1}) and (\ref{eq3.5.1}) are equal, it is easy to check that 
\begin{equation}
\label{eq6.21.1}
A_M(z) \subseteq A
\end{equation}
for given $z$, where $A$ is the event in (\ref{eq3.20.1}) with $x = \sqrt{N_M(z)}$.  It therefore follows
from Proposition \ref{prop3.7.1} that, for appropriate $\aaa$,
\begin{equation}
\label{eq6.21.2}
P(A_M(z)) \le P(A) \le \aaa N_M(z)^{-\frac{5}{2\eta} - 1}.
\end{equation}
%
%
%
%
Together with (\ref{eq2.7.2}), (\ref{eq4.20.3}) and (\ref{eq6.21.2}) imply that, for given $z$,
\begin{equation}
\label{eq6.21.3}
E[Z_k(T_{3}^{\prime}(M))^2; \, A_M(Z(0))| \, Z(0)=z] \le \bbb / \sqrt{N_M(z)} \le \bbb /\sqrt{M}
\end{equation}
for $k=2,\ldots,5$ and appropriate $\bbb$.  Similiarly, by 
(\ref{eq2.7.3}), (\ref{eq4.20.3}) and (\ref{eq6.21.2}),
\begin{equation}
\label{eq6.21.3'}
E[Z_k(T_{3}^{\prime}(M)); \, A_M(Z(0))| \, Z(0)=z] \le \ccc / N_M(z) \le \ccc /M
\end{equation}
for $k=1,6$ and appropriate $\ccc$.  Integrating (\ref{eq6.21.3}) and (\ref{eq6.21.3'}) 
over $z$ produces (\ref{eq4.21.4}) and (\ref{eq4.21.4'}).  
\end{proof}
We now demonstrate Proposition \ref{prop4.8.1}.  We first show (\ref{eq4.8.4}) and (\ref{eq4.8.5}),
which are then used to show (\ref{eq4.8.6}) and (\ref{eq4.8.7}).  On account of the upper bounds on 
$z_1$ and $z_6$ in (\ref{eq4.8.3}), $Z_1(\cdot)$ will drift toward $0$ so that $Z_1(t) = \epsk M_1$ 
typically occurs before time $\ssh M_1$.  This will imply (\ref{eq4.8.4}).  Since $Y_1(T_5(M_1)) = 0$,
the coordinates $k=2,\ldots,6$ drift toward $0$  over $[0,T_5(M_1)]$, implying (\ref{eq4.8.5}).
The elapsed time between $T_5(M_1)$ and $T_{5}^{\prime}(M_1)$ is short enough so $Z_k(T_{5}^{\prime}(M_1))$,
$k=1,6$, will typically still be small, and so (\ref{eq4.8.6}) will hold.  It is also short enough so
$Y_1(T_{5}^{\prime}(M_1)) = 0$ and long enough for $Z_k(t) = 0$, $k=2,\ldots,5$, to typically occur, from which 
(\ref{eq4.8.7}) will follow. 
\begin{proof}[Proof of Proposition \ref{prop4.8.1}]
We first demonstrate (\ref{eq4.8.4}).  
To do so, we analyze $Z(t)$ when $Y_1(t)=0$, for given $t\ge 0$.  
By (\ref{eq1.2.1}),
\begin{equation}
\label{eq6.6.1}
Z_1(t) = Z_1(0) + B_1(t) - t + (1 + \delta_2)\sum_{k=1}^{6}Y_k(t) - (\delta_2 + \delta_4)Y_6(t).
\end{equation}
Since $Y_1(t)=0$, it follows from (\ref{eq2.10.3}) of 
Lemma \ref{lem2.10.2} and (\ref{eq4.8.3}) that
\begin{equation}
\label{eq6.6.3}
Y_6(t) \ge \tfrac{1}{5}t - \tfrac{\ssg}{\delta_3}M_1
\end{equation} 
off of a set of exponentially small probability in $M_1$.  Together,  (\ref{eq1.6.3}), (\ref{eq4.8.3}), 
(\ref{eq6.6.1}) and (\ref{eq6.6.3}) imply that
\begin{equation}
\label{eq6.6.4}
Z_1(t) \le \frac{2 \delta_4 \ssg}{\delta_3}M_1 + B_1(t) - \left(1 + \frac{\delta_2 + \delta_4}{5}\right)t 
+ (1 + \delta_2)\sum_{k=2}^{6}Y_k(t).
\end{equation}

Now, by (\ref{eq2.10.4}) of Lemma \ref{lem2.10.2}, one has that, for given $\epsilon > 0$, 
\begin{equation}
\label{eq6.7.1}
\begin{split}
 - \left(1 + \frac{\delta_2 +  \delta_4}{5}\right)t + (1 & + \delta_2)\sum_{k=2}^{6}Y_k(t) \\
& \le \left[\frac{(1+\epsilon)(1+\delta_2)}{1-\delta_3} - 1 - \frac{\delta_2 + \delta_4}{5}\right]t
\end{split}
\end{equation}
off of a set of exponentially small probability in $M_1$.  One can check that, because of (\ref{eq1.6.3})
and $\delta_3 \le \tfrac{1}{10}$, the right side of (\ref{eq6.7.1}) is less 
than $-(\tfrac{1}{5}\delta_2 + \epsilon)t$
for $\epsilon$ chosen small enough.  Combining (\ref{eq6.6.4}), (\ref{eq6.7.1}) and applying (\ref{eq2.1.2'}) of
Lemma \ref{lem2.1.1}, one therefore obtains
\begin{equation}
\label{eq6.7.2}
Z_1(t) < \frac{2\delta_4 \ssg}{\delta_3}M_1 - \frac{1}{5} \delta_2 t%
\end{equation}
off of a set of exponentially small probability in $M_1$, provided $Y_1(t) = 0$.  But, since $Z_1(t) \ge 0$, it
follows from (\ref{eq6.7.2}) that, off this set, $Y_1(\ssh M_1) > 0$ for 
$\ssh = (10\ssg \delta_4/\delta_2\delta_3)M_1$.  So, $Z_1(t) = 0$ for some $t\le \ssh M_1$, which implies
(\ref{eq4.8.4})

Set $\tau_k = \min\{t: Z_k(t) = 0\}$ for $k=2,\ldots,6$.  In order to demonstrate (\ref{eq4.8.5}), we 
show that 
\begin{equation}
\label{eq6.8.1}
\tau_k \le T_5(M_1) \qquad \text{for } k=2,\ldots,6,
\end{equation}
off of a set of exponentially small probability in $M_1$.  Inequality (\ref{eq4.8.5}) then follows from
(\ref{eq4.8.4}) and (\ref{eq2.5.4}) of Lemma \ref{lem2.5.2} for a small enough choice of $\epsilon > 0$.

We claim that, off of a set of exponentially small probability in $M_1$, 
\begin{equation}
\label{eq6.8.1'}
\tau_6 < T_5(M_1).
\end{equation}
To see this, note that, when $Y_6(t)=0$, it follows from (\ref{eq6.6.1}), (\ref{eq2.3.3}) of Lemma \ref{lem2.3.2},
(\ref{eq4.8.2})--(\ref{eq4.8.3}) and $\epsq \ge 3\epsp$ that
\begin{equation}
\label{eq6.6.2}
Z_1(t) \ge \tfrac{1}{3} \epsq M_1 +\delta_2 t + (B_1(t) - (1+\delta_2)B_2(t)).
\end{equation}
Since $\epsq \ge 6\epsk$, on account of (\ref{eq2.1.3}) of Lemma \ref{lem2.1.1}, this is greater than 
$\epsk M_1$ off of a set of exponentially small probability in $M_1$.  Together with (\ref{eq4.8.4}), this implies the claim.

Also note that, for each $k=2,\ldots,5$, it follows from (\ref{eq1.2.1}) and (\ref{eq4.8.2})--(\ref{eq4.8.3})
that
$$ Z_k(t) - Z_6(t) \le B_k(t) - B_6(t) - \tfrac{2}{3}\epsq M_1 $$
on $Y_k(t) = 0$.  Together with (\ref{eq2.1.2'}) of Lemma \ref{lem2.1.1}, this implies
\begin{equation}
\label{eq6.8.2}
\tau_6 \ge \tau_k \wedge \ssh M_1 \ge \tau_k \wedge T_5(M_1) \qquad \text{for } k=2,\ldots,5,
\end{equation}
off of a set of exponentially small probability in $M_1$.  
Together with (\ref{eq6.8.1'}), this implies (\ref{eq6.8.1}).

In order to demonstrate (\ref{eq4.8.6}) for $k=1$ and $k=6$, we restart $Z(\cdot)$ at time $T_5(M_1)$
and then apply (\ref{eq2.4.4}) of Lemma \ref{lem2.4.2}.  Using the definition of $T_5(M_1)$, for $k=1$,
and (\ref{eq4.8.5}), with $\epsr \stackrel{\text{def}}{=}\tfrac{1}{20}\delta_3 \epsk \le \epsk$, for $k=6$,
one obtains
\begin{equation}
\label{eq6.8.3}
P(Z_k(T_{5}^{\prime}(M_1)) \ge 5\epsk M_1) \le \ssk e^{-\epsilon^{\prime}M_1} \qquad \text{for } k=1,6,%
\end{equation}
for appropriate $\ssk$ and $\epsilon^{\prime} > 0$.  Inequality (\ref{eq4.8.6}) then follows from
(\ref{eq2.7.3}) of Lemma \ref{lem2.7.1}

In order to demonstrate (\ref{eq4.8.7}), we restart $Z(\cdot)$ at time $T_5(M_1)$, denoting the new 
process by $\tilde{Z}(\cdot)$ and the corresponding Brownian motion by $\tilde{B}(\cdot)$.  By (\ref{eq1.2.1}),
$$ \tilde{Z}_1(t) \ge \epsk M_1 - t + \tilde{B}_1(t), $$
and so $\tilde{Z}_1(t) > 0$ on $[0,\tfrac{1}{2}\epsk M_1]$ off of a set of exponentially small
probability in $M_1$.  Since $\epsr < \tfrac{1}{10} \delta_3 \epsk$, it follows, by (\ref{eq2.10.3}) of
Lemma \ref{lem2.10.2} and (\ref{eq4.8.5}), that
\begin{equation}
\label{eq6.9.1}
\tilde{Y}_k(\tfrac{1}{2}\epsk M_1) \ge (\tfrac{1}{10}\epsk - \tfrac{\epsr}{\delta_3})M_1 > 0 
\qquad \text{for } k=2,\ldots,5,
\end{equation}
off of a set $\ffk$ of exponentially small probability in $M_1$.

We denote by $\tilde{\tau}_k$ the first time at which $\tilde{Z}_k(t) = 0$.  Restarting the process at
$\tilde{\tau}_k$, it follows from (\ref{eq2.6.1}) of Lemma \ref{lem2.5.2} that
$$ E[\tilde{Z}_k(\tfrac{1}{2}\epsk M_1)^2; \, \ffk^c] \le 24\cdot8 \epsk M_1 \qquad \text{for } k=2,\ldots,5. $$
On account of the upper bounds on $P(\ffk)$ and (\ref{eq2.7.2}) of Lemma \ref{lem2.7.1}, one obtains
\begin{equation}
\label{eq6.9.2}
E[\tilde{Z}_k(\tfrac{1}{2}\epsk M_1)^2] \le 24\cdot8 \epsk M_1 + \ssj \qquad \text{for } k=2,\ldots,5
\end{equation}
for appropriate $\ssj$, which depends on $\epsk$ but not on $M_1$.  This implies (\ref{eq4.8.7}).
\end{proof}
\end{document}